\documentclass[11pt]{article}

\usepackage{amsmath,amssymb,graphicx,mathtools}
\numberwithin{equation}{section}
\usepackage{amsthm}
\usepackage{relsize}
\usepackage[margin=1in]{geometry} 
\usepackage{enumerate}
\usepackage{yhmath}
\usepackage{tikz}
\usepackage{hyperref}
\usepackage[flushleft]{threeparttable}
\usepackage{array,booktabs,makecell}

\def\@abssec#1{\vspace{.05in}\footnotesize \parindent .2in
{\bf #1. }\ignorespaces}

\graphicspath{{/EPSF/}{Figures/}}
\DeclareGraphicsExtensions{.eps}

\newtheorem{theorem}{Theorem}[section]
\newtheorem{lemma}[theorem]{Lemma}
\newtheorem{proposition}[theorem]{Proposition}

\newtheorem{definition}[theorem]{Definition}
\newtheorem{remark}[theorem]{Remark}
\newtheorem{hypothesis}[theorem]{Hypothesis}

\newcommand{\gb}[1]{{{\color{blue}{GB: #1}}}}

\def \Rm {\mathbb R}
\def \Nm {\mathbb N}
\def \Cm {\mathbb C}
\def \Zm {\mathbb Z}

\newcommand{\eps}{\varepsilon}
\newcommand{\E}{\mathbb E}
\newcommand{\dsum}{\displaystyle\sum}
\newcommand{\dint}{\displaystyle\int}

\newcommand{\aver}[1]{\langle {#1} \rangle}

\newcommand{\mB}{\mathcal B}

\newcommand{\mF}{\mathcal F}
\newcommand{\mH}{\mathcal H}

\newcommand{\mD}{\mathfrak D}

\newcommand{\fa}{{\mathfrak a}}

\newcommand{\fS}{{\mathfrak S}}
  \newcommand{\rJ}{{\rm J}}

\newcommand{\cout}[1]{}

\newcommand{\rME}{{\rm M}(E)}

\newcommand{\R}{{\rm R}}

\newcommand{\Tr}{{\rm Tr}}

\newcommand{\trr}{{\rm tr}}

\newcommand{\fco}{\Xi}

 \renewcommand{\arraystretch}{1.5}

\title{Scattering theory of topologically protected edge transport}

\author{ Binglu Chen \thanks{Department of Mathematics, University of Chicago, Chicago, IL 60637; {\tt blchen@uchicago.edu}} \and  Guillaume Bal \thanks{Departments of Statistics and Mathematics and CCAM, University of Chicago, Chicago, IL 60637; guillaumebal@uchicago.edu}  }

\begin{document}
 
\maketitle

\begin{abstract}
 This paper develops a scattering theory for the asymmetric transport observed at interfaces separating two-dimensional topological insulators. Starting from the spectral decomposition of an unperturbed interface Hamiltonian, we present a limiting absorption principle and construct a generalized eigenfunction expansion for perturbed systems. We then relate a physical observable quantifying the transport asymmetry to the scattering matrix associated to the generalized eigenfunctions. In particular, we show that the observable is concretely expressed as a difference of transmission coefficients and is stable against perturbations.  We apply the theory to  systems of perturbed Dirac equations with asymptotically linear domain wall. 
\end{abstract}

\renewcommand{\thefootnote}{\fnsymbol{footnote}}
\renewcommand{\thefootnote}{\arabic{footnote}}

\renewcommand{\arraystretch}{1.1}



\noindent{\bf AMS:} 35P25, 81Q05, 81U20, 47A53.


%
\section{Introduction}
%
Topological insulators find applications in many areas in condensed matter physics, photonics, and geophysical sciences; see, e.g.,   \cite{BH, delplace, sato, Volovik, Witten} for pointers to a vast literature.  A characteristic feature of such two-dimensional systems is the topologically protected asymmetric transport observed along one-dimensional interfaces separating two-dimensional insulating bulks. 

If $H$ is a Hamiltonian describing transport in such a two-dimensional system on the Euclidean plane $\Rm^2$, the asymmetry along the edge is captured by the following physical observable:
\begin{equation}\label{eq:edgeconductivity}
  \sigma_I[H] = {\rm Tr} \, i[H,P]\varphi'(H).
\end{equation}
Here, $P=P(x)\in\fS[0,1]$ and $\varphi(E)\in \fS[0,1,E_-,E_+]$  are smooth switch functions, where $\fS[a,b,c,d]$ is the set of bounded (measurable) functions on $\Rm$ equal to $a$ for $x<c$ and equal to $b$ for $x>d$ while $\fS[a,b]$ is their union over (finite) $c<d$.  Tr is operator trace on an appropriate Hilbert space. The function $P(x)$ heuristically models projection onto a half space $x>x_0$ for some arbitrary $x_0\in\Rm$ with $(x,y)$ spatial coordinates in $\Rm^2$, while the operator $i[H,P]$ may then be interpreted as a current operator (rate of current flow from the left to the right of the vertical line $x=x_0$ per unit time). Finally, $0\leq \varphi'(E)$ captures a density of states of interest.  We assume $\varphi'(E)$ supported in an energy interval $[E_-,E_+]$ with $\varphi(E_-)=0$ while $\varphi(E_+)=1$. The interval $[E_-,E_+]$ is such that propagation for such an energy range only occurs along the interface $y\approx0$. In other words, a confining mechanism encoded in the Hamiltonian $H$ implies a suppressed propagation into the bulk for this energy range. Thus $\sigma_I$ models the expected value of the current operator for excitations in the system with density $\varphi'(E)$ supported in the energy interval $[E_-,E_+]$.

The interface current observable first introduced in \cite{SB-2000} has been used in a variety of contexts in the study of topological insulators \cite{BH, Drouot, Elbau, elgart2005equality}. See \cite{2, 3, bal2023topological,quinn2021approximations} for an analysis of \eqref{eq:edgeconductivity} when $H$ is a differential operator acting on functions on $\Rm^2$, which is the setting we consider in this  paper. In all these instances, one can prove that $2\pi\sigma_I\in\Zm$ is quantized and describes a transport asymmetry when it is non-vanishing.

\medskip

The main objective of this paper is to construct a quantitative theory for such an interface transport and to understand how the latter is affected by the non-triviality of $2\pi\sigma_I\in\Zm$. The theoretical results obtained in the paper then serve as foundations for robust and accurate numerical simulations of the interface transport, as performed in, e.g., \cite{bal2023asymmetric,bal2023mathbb}. 

We consider the following setting.  Let $H_0$ be an unperturbed operator invariant with respect to spatial translations in $x$ so that plane waves along the interface may be identified as generalized eigenfunctions of $H_0$. We will show how $\sigma_I$ may be expressed in terms of such eigenfunctions \cite{bal2023asymmetric}.  For elliptic operators $H_0$, the topological protection of the asymmetric transport states \cite{2, 3,quinn2021approximations} that $\sigma_I[H_0]=\sigma_I[H_0+Q]$ for a large class of perturbations $Q$. The asymmetric transport may thus be interpreted as an obstruction to Anderson localization \cite{1,PS}: no matter how large $Q$ is (so long as it is localized spatially), some transmission is guaranteed by the non-vanishing current $2\pi\sigma_I\not=0$. Such an obstruction is analyzed numerically in \cite{bal2023mathbb}.

This paper introduces a scattering theory for $H=H_0+Q$ in the setting where the spectral decomposition of $H_0$ is {\em known} and $Q$ is a sufficiently short-range perturbation. More precisely, for an energy $E\in[E_-,E_+]$ , we wish to show the existence of generalized plane waves solutions of $H\psi_m=E\psi_m$ and construct a scattering matrix $S=S(E)$ from such functions $\psi_m$. We also aim to show that $2\pi\sigma_I\in\Zm$ is directly related to the coefficients of the scattering matrix $S$, see Theorem \ref{thm:scatmatrix} below. 

The spectral analysis of perturbations $H_0+Q$ of constant-coefficient differential operators $H_0$ in any spatial dimension is well developed; see, e.g., \cite{ASNSP_1975_4_2_2_151_0,H-II-SP-83,simon1982schrodinger,teschl2014mathematical}. Much less is known when $H_0$ is modeling such interface transport, as an operator on functions of $\Rm^2$ while transport is only allowed along a one-dimensional submanifold. The simplest such operator is arguably $H_0=-\partial_x^2-\partial_y^2+ y^2$, for which the theory presented in this paper applies. We construct such a scattering theory based on the limiting absorption principle following \cite{ASNSP_1975_4_2_2_151_0}. This will allow us to show that the perturbed operator $H_0+Q$ has a spectral decomposition based on generalized eigenfunctions out of which an energy-dependent unitary scattering matrix may be defined and related to the edge current observable $\sigma_I$.

The above spectral theory requires the unperturbed operator $H_0$ to satisfy a number of hypotheses. We show in this paper that the hypotheses are satisfied when $H_0$ is a Dirac operator with an unbounded domain wall. This is the setting considered in the numerical simulations in \cite{bal2023asymmetric,bal2023mathbb}. The limiting absorption principle and the theory eigenfunction expansions of \cite{ASNSP_1975_4_2_2_151_0} have been implemented in detail for local perturbations of constant-coefficient Dirac operators in \cite{yamada, yamada1975eigenfunction}. Our analysis of Dirac operators with domain walls closely follows the structure presented in these papers. References on scattering theory, the limiting absorption principle, and generalized eigenfunction expansions that are relevant to the current work include \cite{ASNSP_1975_4_2_2_151_0,H-II-SP-83,ikebe1960eigenfunction,kato2013perturbation,RS4,simon1982schrodinger,teschl2014mathematical,yamada,yamada1975eigenfunction}.

\medskip

The rest of the paper is structured as follows. Section \ref{sec:main} presents the main hypotheses required on the unperturbed operator $H_0$ and on the perturbation $Q$ and the main results of the paper. The scattering theory for $H$ under such hypotheses and its relation to $2\pi\sigma_I$ is given in section \ref{sec:current}. The proof of the limiting absorption principle necessary to develop the scattering theory is given in section \ref{sec:lap}, and in particular guarantees the discreteness of the point spectrum and the absence of singular continuous spectrum.   The derivation of the eigenfunction expansion for the perturbed operator is given in section \ref{sec:eigexp}.  Finally, the detailed verification of all necessary hypotheses for a class of Dirac operators is carried out in section  \ref{sec:Dirac}.

%
\section{Main assumptions and main results}\label{sec:main}
%
This section first describes the main hypotheses required to define a scattering theory.

Let $\mH=L^2(\Rm^2)\otimes\Cm^q$ be the Hilbert space of vector-valued functions with square-integrable entries defined on the Euclidean plane with coordinates $(x,y)$. We define the following spaces.
\begin{definition}\label{def:LHs}
For $s\in\Rm$ and $p\geq0$, we define the Hilbert spaces $L^2_{s}$ and $H^p_s$ as the completion of $C^\infty_c(\Rm^2)$ for the norms
\begin{align}\label{eq:Hsp} \nonumber
\|u\|_{L^2_{s}} &\ = \ \Big(\int_{\mathbb{R}^2}\aver{x}^{2s} |u(x,y)|^2dxdy\Big)^{\frac{1}{2}} \\
\|u\|_{H_s^p} & \ = \  \Big(\int_{\mathbb{R}^2}\Big[\aver{x}^{2s}\aver{y}^{2p}|u(x,y)|^2+\aver{x}^{2s}\sum_{|\alpha|= p}|D^\alpha u(x,y)|^2\Big]dxdy\Big)^{\frac{1}{2}}.
\end{align}
Here, $\aver{x}:=\sqrt{1+x^2}$. We denote $H^p = H^p_0$. We also denote by $H^p_s$ ($L^2_s$) the space of vector-valued functions $H^p_s\otimes \Cm^q$ ($L^2_s\otimes\Cm^q$).
\end{definition}

We then start with an unperturbed self-adjoint operator $H_0$ from $\mD(H_0)=H^p$ to $\mH$ that is invariant with respect to translations in $x$. Thus, $H_0=\mF_{\xi\to x}^{-1} \hat H_0(\xi) \mF_{x\to\xi}$ with $\mF_{x\to\xi}$ Fourier transform in the first variable $x$ (see \eqref{eq:unperturbedxi} for convention) and $\Rm\ni \xi\mapsto \hat H_0(\xi)$ a family of self-adjoint operators on $L^2(\Rm)\otimes\Cm^q$.  For $H_0$ an elliptic operator of order $p$, which is the framework of interest in this paper, standard ellipticity results show that the generalized eigenfunction  $\psi_j(x,y;\xi)$ defined in \eqref{eq:unperturbedxi} below is an element in $H^p_{-s}$ for $s>\frac12$. We have by assumption the spectral decomposition
\begin{equation}\label{eq:spectraldecH0}
  \qquad   H_0  = \dsum_{j\in J} \dint_{\Rm} E_j(\xi) \Pi_j(\xi) d\xi,\qquad \Pi_j(\xi) = \psi_j(\cdot;\xi) \otimes \psi_j(\cdot;\xi).
\end{equation}
Associated to the above decomposition is the following resolution of identity. Let $\fco=(j,\xi)\in J\times \Rm$, where $J\cong \Nm$. We define for $f\in L^2(\Rm^2)\otimes\Cm^q$ the (unperturbed) Fourier transform:
\begin{equation}\label{eq:FT}
    \hat f (\Xi) =  (\mF f)(\Xi) := \dint_{\Rm^2} \overline{\psi_j(x,y,\xi)} \cdot f(x,y) dxdy = (f,\psi_j(\cdot,\xi)),
\end{equation}
where $(f,g)=\int_{\Rm^2} f(x,y)\cdot \bar g(x,y) dxdy$ is the inner product on $\mH$,
with inverse Fourier transform:
\begin{equation}\label{eq:IFT}
    f (x,y) =  (\mF^{-1} \hat f)(x,y) := \dsum_j \dint_{\Rm} \hat f(\Xi)  \psi_j(x,y,\xi) d\xi.
\end{equation}
The Fourier transform is an isometry from $\mH=L^2(\Rm^2,dxdy)\otimes\Cm^q$ to $L^2(J\times \mathbb{R},d\fco;\Cm)$, with $d\fco$ the Cartesian product of the counting measure on $J$ and the Lebesgue measure on $\Rm$.

The main hypotheses we request on $H_0$ are summarized as follows:

\begin{hypothesis}[${\rm [H1]}$] \label{hyp:H1}
(o) We assume that $H_0$ is a self-adjoint (elliptic) differential operator with domain $H^p\otimes \Cm^q$ and resolvent operator $R_0(i)=(H_0-i)^{-1}$ bounded from $\mH$ to $H^p\otimes \Cm^q$.
\\[1mm]
(i) For each $\xi\in\Rm$, $\hat H_0(\xi)$ has a compact resolvent and hence purely discrete spectrum. We assume the existence of generalized eigenfunctions in $H^p_{-s}$ for $s>\frac12$, solutions
\begin{equation}\label{eq:unperturbedxi}
 \psi_j(x,y;\xi) = \frac{1}{\sqrt{2\pi}}  e^{i\xi x} \phi_j(y;\xi),
\end{equation}
of the eigenvalue problem $(H_0-E_j(\xi))\psi_j=0$ with $(\phi_j)_j$ an orthonormal basis of $L^2(\Rm_y)\otimes\Cm^q$, i.e., $(\phi_j,\phi_k)_{L^2(\Rm_y)\otimes\Cm^q}=\delta_{jk}$. Here, $j\in J \cong \Nm$. 
\\[1mm]
(ii) We assume that the branches of absolutely continuous spectrum $j\mapsto E_j(\xi)$ are smooth and satisfy $|E_j(\xi)|\to\infty$ as $|\xi|\to\infty$ with $\xi\mapsto (1+|E_j(\xi)|^2)^{-1}$ integrable for $j\in J$. We assume that for any interval $[a,b]$, only a finite number of branches $\xi \mapsto E_j(\xi)$ cross $[a,b]$. 
\\[1mm] (iii)
The spectral elements $E_j(\xi)$ and $\Pi_j(\xi)$ are assumed to be smooth in $\xi$ with a finite number of critical values. Define
\begin{equation}\label{eq:Z}
  Z = \big\{ E \in \Rm ;\  E=E_j(\xi) \mbox{ for some } (j,\xi)\in J\times\Rm \mbox{ and } \partial_\xi E_j(\xi)=0 \big\}.
\end{equation}
We assume the set $Z$ of critical values to be  finite in each bounded interval $[E_-,E_+]$.
\\[1mm]
(iv) To set up a scattering theory, we finally assume the following completeness property: for any $E\in \Rm\backslash Z$, then any solution $\psi\in H^p_{-s}$ of $(H_0-E)\psi=0$ is a linear combination of the generalized eigenfunctions $\psi_j(x,y;\xi)$ for values of $\xi$ such that $E_j(\xi)=E$. We label $\psi_m(x,y;E)=\psi_j(x,y;\xi_m)$ for $1\leq m\leq \rME$ the corresponding solutions at $E$ fixed.  Up to (obvious) relabeling, we thus have
\begin{equation}\label{eq:unperturbedE}
  \psi_m(x,y;E) =\frac{1}{\sqrt{2\pi}}  e^{i\xi_m(E) x} \phi_m(y;\xi_m(E)),\qquad 1\leq m\leq \rME.
\end{equation}\end{hypothesis}

Whereas the spectral representation \eqref{eq:spectraldecH0} defines spectral branches $\Xi\mapsto E(\Xi)\equiv E_j(\xi)$, scattering theory requires the inverse map $E\mapsto \Xi(E)\equiv \xi_m(E)$ for $1\leq m\leq \rME$. This implicitly defines $\rME$ as the (locally finite) number of generalized eigenvectors corresponding to a given energy level $E$; see Fig.\ref{fig:1} for an illustration for Dirac operators with linear domain walls. 

\medskip

Consider now a perturbed operator $H=H_0+Q$ where $Q$ is a short-range operator. We consider the case where $Q$ is an operator of multiplication by $Q(x,y)$ with the $q\times q$-valued (measurable) function $Q(x,y)$ such that $\aver{x}^{1+\eps} |Q(x,y)|\leq C$ uniformly in $(x,y)\in\Rm^2$ for some $\eps>0$. We thus assume $Q$ sufficiently rapidly decaying (only) in the $x$ variable. 

We next assume for each $\Xi=(j,\xi)$ the existence of modified generalized eigenfunctions $\psi^Q_j\in H^p_{-s}$ solution of $H\psi^Q_j=E_j(\xi)\psi^Q_j$ such that the following spectral decomposition holds:
\begin{equation}[{\rm H2}]\label{eq:spectraldecH}
   \qquad H  = \dsum_n \lambda_n \Pi_n + \dsum_j\dint_{\Rm} E_j(\xi) \Pi^Q_j(\xi) d\xi,\quad \Pi^Q_j(\xi) = \psi^Q_j(\cdot;\xi) \otimes \psi^Q_j(\cdot;\xi).
\end{equation}
Here, the sum over $n$ corresponds to point spectrum, which is discrete (locally finite) away from possibly a discrete set of essential (point) spectrum, with eigenvalues $\lambda_n$ and (rank-one) projectors $\Pi_n = \psi_n\otimes \psi_n$ for eigenfunctions $\psi_n\in\mH$. The above decomposition thus implies that the projectors $\Pi_n$ and $\Pi^Q_j(\xi)d\xi$ form a (complete) resolution of identity.  It imposes that the branches of absolutely continuous spectrum for $H$ and $H_0$ be the same. This is consistent with the standard result that two operators $H_1$ and $H_2=H_1+V$ with $V$ a trace-class perturbation have unitarily equivalent absolutely continuous spectrum \cite{kato2013perturbation}. While $H_0$ has only absolutely continuous spectrum, the perturbation $H=H_0+Q$ may possess discrete point spectrum. The justification of [H2] in practice will come from obtaining a generalized Fourier transform theory where $\psi_j$  in \eqref{eq:FT} and \eqref{eq:IFT} is replaced by the generalized plane wave $\psi^Q_j$.

\medskip

Finally, in order to be able to construct scattering matrices at a fixed level $\Rm\ni E\not\in Z$, we still denote by $\psi^Q_m$ the solution of the problem $(H-E)\psi^Q_m=0$ with $\psi^Q_m = \psi^Q_j(\cdot,\xi_m)$ for $E_j(\xi_m)=E$ and $1\leq m\leq \rME$. We now impose the final assumption on these generalized eigenfunctions that for $|x|$ large, then $\psi^Q_m$ is approximately given by a linear combination of the unperturbed solutions $\psi_n$ for $1\leq n\leq \rME$. More precisely, we assume that for $1\leq m\leq\rME$,
\begin{equation}[{\rm H3}]\label{eq:approxscattering}
  \qquad \psi_m^Q(x,y) \approx  \dsum_{1\leq n\leq \rME} \alpha^\pm_{mn} \psi_n(x,y) =  \dsum_{1\leq n\leq \rME} \alpha^\pm_{mn} \frac{1}{\sqrt{2\pi}} e^{i\xi_n x}  \phi_n(y) ,
\end{equation}
for some coefficients $\alpha_{mn}^\pm\in\Cm$, where $a \approx b$ means that the difference $a-b$ converges to $0$ uniformly (in $x$ as a square-integrable function in $y\in\Rm$) as $x\to\pm\infty$ and where $\alpha^\pm$ are the corresponding coefficients in these two limits.

\paragraph{Scattering matrix and edge current observable.} Let us assume a system for which [H1], [H2] and [H3] above hold. In particular, let $E\in\Rm\backslash Z$ be fixed and $\rME$ be the corresponding number of generalized eigenvectors. Define the currents
\begin{equation}\label{eq:currents}
   J_m= J_m(E) = \partial_\xi E_m(\xi_m) \not=0
\end{equation}
which do not vanish for $E\not\in Z$ not being a critical value of the branches of absolutely continuous spectrum. 

We may then define the following reflection and transmission coefficients $R^\pm_{mn}$ and $T^\pm_{mn}$ as
\begin{equation}\label{eq:RTmn}
  \begin{array}{rcll}
  T^+_{mn} &=& \sqrt{\frac{|J_n|}{|J_m|}}\ \alpha^+_{mn} \ \ &\mbox{when} \ \ J_m>0 \mbox{ and } J_n>0\\[2mm]
   T^-_{mn} &=& \sqrt{\frac{|J_n|}{|J_m|}}\ \alpha^-_{mn} \ \ &\mbox{when} \ \ J_m<0 \mbox{ and } J_n<0\\[2mm]
   R^+_{mn} &=&  \sqrt{\frac{|J_n|}{|J_m|}}\ \alpha^-_{mn} \ \ &\mbox{when} \ \ J_m<0 \mbox{ and } J_n>0\\[2mm]
    R^-_{mn} &=&  \sqrt{\frac{|J_n|}{|J_m|}}\ \alpha^+_{mn} \ \ & \mbox{when} \ \ J_m>0 \mbox{ and } J_n<0.
  \end{array}
\end{equation}
Let then $T_+$ be the $n_+\times n_+$ matrix of coefficients $T^+_{mn}$ above, while $T_-$ is the $n_-\times n_-$ matrix of coefficients $T^-_{mn}$, $R_+$ is the $n_-\times n_+$ matrix of coefficients $R^+_{mn}$, and finally $R_-$ is the $n_+\times n_-$ matrix of coefficients $R^-_{mn}$.  We note for completeness that we also have $\alpha^-_{mm}=1$ when $J_m>0$ and $\alpha^+_{mm}=1$ when $J_m<0$. All other coefficients $\alpha^\pm_{ij}$ not mentioned above then vanish.

Here $n_+=n_+(E)$ is the number of generalized eigenfunctions $\psi_m$ with $J_m>0$ while $n_-=n_-(E)$ is the number of generalized eigenfunctions $\psi_m$ with $J_m<0$. Thus, $n_++n_-=\rME$. We then have the following result.
\begin{theorem}\label{thm:scatmatrix}
 The $(n_++n_-)\times(n_++n_-)$ scattering matrix
 \begin{equation}\label{eq:scatteringmatrix}
  S = S(E) = \begin{pmatrix} T_+ & R_- \\R_+&T_-\end{pmatrix}
\end{equation}
is unitary. Let $\varphi'(H)$ be supported in $[E_-,E_+]$ such that $E\in[E_-,E_+]$and $[E_-,E_+]\cap Z=\emptyset$. Then, the interface current observable may be recast as 
 \[
  2\pi \sigma_I =  \trr\ T^*_+T_+ - \trr\ T^*_-T_- = n_+-n_-.
 \]
\end{theorem}
This result shows that the edge current observable, which is independent of the perturbation $Q$ and hence given by $n_+-n_-$ that may directly be read off the unperturbed operator $H_0$, is also given by the excess of transmission in one direction compared to the other direction $\trr\ T^*_+T_+ - \trr\ T^*_-T_-$, which is intuitively reasonable. Note that the coefficients in $S(E)$ strongly depend on the perturbation $Q$ and the value $E$ as may be seen in the numerical simulations in \cite{bal2023asymmetric,bal2023mathbb}. The proof of this result is presented in section \ref{sec:current}.

\paragraph{Limiting absorption principle.}

Assume that the unperturbed operator $H_0$ satisfies [H1]. Under additional assumptions on $H_0$ and on the perturbation $Q$, we show that [H2] holds for a large class of problems. 

We assume that $H_0$ is a self-adjoint differential operator as described in [H1] above. The resolvent operator $R_0(z)=(H_0-z)^{-1}$ may then display different behaviors as $z$ approaches the real-axis with positive or negative imaginary parts. 
\begin{definition}\label{def:Jab}
For $a<b$, we define
\begin{align*}
    \rJ_+(a,b)&=\{\lambda\in\mathbb{C}\ |  \ a<\operatorname{Re}\lambda<b,\  0<\operatorname{Im} \lambda<1\},\\
    \rJ_-(a,b)&=\{\lambda\in\mathbb{C}\ |\  a<\operatorname{Re}\lambda<b, \ -1<\operatorname{Im} \lambda<0\}, \\
    \rJ(a,b)&=\rJ_+(a,b)\cup \rJ_-(a,b).
\end{align*}
\end{definition}
Our objective is to prove results on the spectrum of $H$ that will allow us to verify hypothesis [H2]. 
We recall that the spaces $L^2_s$ and $H^p_s$ are introduced in Definition \ref{def:LHs} and that $Z$ is the set of critical values of branches of spectrum of $H_0$ defined in \eqref{eq:Z}.  

We make the following assumptions, recalling that $\aver{x}=\sqrt{1+x^2}$:
\begin{hypothesis}\label{hyp:Q}
    We assume that $Q(x,y)$ is a $q\times q$ Hermitian matrix valued function. Moreover,  $|Q(x,y)|$ is bounded (measurable) and for some $h>1$ and $C=C(h)>0$,
        \begin{gather}\label{q}
|Q(x,y)| \leq C \aver{x}^{-h} ,\qquad (x,y)\in\Rm^2.
        \end{gather}
\end{hypothesis}
The above may be generalized to the short-range condition in Remark \ref{rem:sr}.
\begin{hypothesis}\label{hyp:H0H} Let $\Rm\ni a<b\in \Rm$. We assume the following a priori estimates.
\begin{itemize}
 \item[1.] Let $s>\frac12$. There is a constant $C=C(s,a,b)>0$ such that 
 \begin{gather}\label{eq:H_0}
    \| u \|_{H_{-s}^p}\leq C\|(H_0-\lambda)u\|_{L_s^2},
\end{gather}
for all complex numbers $\lambda\in \rJ(a,b)$ and $u\in H_s^p$.
\item[2.] Let $s>0$, $\epsilon>0$, and $(a,b)\cap Z=\emptyset$. There is a constant $C=C(s,a,b)>0$ such that 
\begin{gather}\label{eq:H0real}
    \| u \|_{H_{s-1-\epsilon}^p}\leq C\|(H_0-\lambda)u\|_{L_s^2},
\end{gather}
for all real numbers $\lambda\in (a,b)$ and $u\in H^p$.
\end{itemize}
\end{hypothesis}
Under these two hypotheses, we then have the following result:
\begin{theorem}\label{thm:lap}
  \noindent (i) The point spectrum of $H$ away from $Z$ is discrete.  The only possible limiting points of families of eigenvalues are in  $Z\cup \{\pm \infty\}$.
  \\
  \noindent (ii) [Principle of Limiting Absorption]. Let $a,b\in\Rm$ such that $[a,b]\cap Z=\emptyset$ and $[a,b]$ does not contain any eigenvalue of $H$. For $1<2s<h$, $f\in L_s^2$ and $\operatorname{Im}z\neq 0$, define 
\begin{gather*}
    u_z(f)=(H-z)^{-1}f.
\end{gather*}
Then for $\lambda\in (a,b)$, there exist $u^\pm(\lambda,f)$ such that
\begin{gather*}
    u_z(f)\to u^\pm(\lambda,f)\ \text{in}\ H^p_{-s}
\end{gather*}
as $z\to\lambda\pm 0i$, respectively. Moreover, $u^\pm(\lambda,f)$ are solutions of the equation 
\begin{gather*}
    (H-\lambda)u=f
\end{gather*}
and they are continuous functions of $\lambda$ in the topology of $H_{-s}^p$.
\\
\noindent (iii) $H$ does not have singular continuous spectrum.
\end{theorem}
The proof of this theorem is presented in section \ref{sec:lap} following the theory of \cite{ASNSP_1975_4_2_2_151_0}; see also \cite{yamada}.

\paragraph{Eigenfunction expansion.} We are now in a position to justify [H2] under the hypotheses leading to Theorem \ref{thm:lap}, and in particular to construct generalized eigenfunctions $\psi_j^Q(\cdot;\xi)$ for $\Xi=(j,\xi)\in J\times \Rm$. We now know that the point spectrum $(\lambda_n)_n$ of $H$ is discrete away from $Z$. Let $\{\varphi_n\}$ be the at most countable set of orthonormal eigenfunctions of $H$. Define $Z_H=Z \cup \{(\lambda_n)_n\}$.  Let $E\in (a,b)$ with $[a,b]\cap Z_H=\emptyset$.

From Theorem \ref{thm:lap}, $R(z)=(H-z)^{-1}$ is well defined on $\mH$ and bounded uniformly for $z\in \rJ(a,b)$ in an appropriate topology as well as the limiting operators:
\[R^\pm(\lambda)=(H-(\lambda\pm i0))^{-1}.\]  

For $\fco=(j,\xi)\in J\times\Rm$, we defined $\psi_j(x,y;\xi)$ in \eqref{eq:unperturbedxi}. For $z\in \rJ(a,b)$, we introduce the function $A_\fco(x,y;z)$ defined by
\begin{equation}\label{eq:Az}
     A_\fco(z) =(I-R(z)Q)\psi_j(\xi).
\end{equation}
Associated is the following linear form defined for $f\in L^2_s$ with $s>\frac12$:
\begin{equation}\label{eq:Azstar}
     A^*_\fco f(z) =(f,A_\fco(z)) := \dint_{\Rm^2} f(x,y) \cdot \bar A_\fco(x,y;z) dxdy.
\end{equation}
We now define the {\em perturbed generalized eigenfunctions}
\begin{equation}\label{eq:perturbedxi}
  \psi_j^\pm(\xi) = A_\fco(E_j(\xi)\pm i0) = (I-R(E_j(\xi)\pm i0)Q)\psi_j(\xi).
\end{equation}
For concreteness, we define $\psi_j^Q(x,y;\xi)=\psi_j^+(x,y;\xi)$, the {\em outgoing} generalized eigenfunctions, while $\psi_j^-(x,y;\xi)$ corresponds to {\em incoming} generalized eigenfunctions.

Theorem \ref{thm:lap} implies that $\psi_j^\pm(\xi)\in H^p_{-s}(\Rm^2)$. For each $E\in(a,b)$, there is a finite number of wavenumbers $\xi_m$ such that $E=E_m(\xi_m)$. We denote by $\psi_m^\pm(E)$ the corresponding generalized eigenfunctions parametrized by $(m,E)$ rather than $(j,\xi)$.

 Let $f\in L_s^2$ and $(j,\xi)$ so that $E_j(\xi)\in (E_-,E_+)\backslash Z_H$. We define the generalized Fourier transform(s) by
 \begin{gather}\label{def:fourier}
    \tilde{f}^\pm(\fco)=A_\fco^* f (E_j(\xi)\pm 0i) = (f,A_\fco(E_j(\xi)\pm 0i)),
\end{gather}  
and for any $\lambda_n\in \Rm$ an eigenvalue of $H\varphi_n=\lambda_n\varphi_n$,
\begin{gather}\label{def:fourierdisc}
    \tilde{f}_n = (f,\varphi_n).
\end{gather}  
Then, we have the following result:
\begin{theorem}[Eigenfunction Expansion Theorem]\label{thm:EET}
For $f\in L_s^2\subset\mH$, we define $\tilde f$ as either one of the generalized Fourier transforms $\tilde f^\pm$ in \eqref{def:fourier}. Then we have the following Parseval relation:
\begin{gather}\label{eigenfunction_expansion}
    \|f\|_{L^2}=\sum_{n=0}^\infty |(f, \varphi_n)|^2+\int_\mathbb{R}\sum_{j\in J}|\tilde{f}(\Xi)|^2d\xi.
\end{gather}
\end{theorem}
This theorem is proved in section \ref{sec:eigexp}. It justifies [H2] as an application of the spectral theorem since the above Parseval relation shows the completeness of the decomposition \eqref{eq:spectraldecH}.

\paragraph{Application to Dirac operator.}  The main unperturbed operator of interest in this paper is the massive two-dimensional Dirac Hamiltonian
\begin{equation}\label{eq:DiracH0}
  H_0 = D_x\sigma_1+D_y\sigma_2 + m(y)\sigma_3, \qquad \hat H_0(\xi) = \xi \sigma_1+D_y\sigma_2 + m(y)\sigma_3,
\end{equation}
with $\sigma_{1,2,3}$ standard Pauli matrices,  $D_x=-i\partial_x$ and $D_y=-i\partial_y$, and $m(y)$ a domain wall, which for concreteness, equals $y$ up to a bounded perturbation. Then, $p=1$ and $q=2$. That the spectral decomposition \eqref{eq:spectraldecH0} and all assumptions in Hypothesis [H1] applies to $H_0$ will be revisited in section \ref{sec:decH0}; see also \cite{bal2023asymmetric}.

A second natural application of the theory developed here is for the Klein-Gordon operator
\begin{equation}\label{eq:KGH0}
  H_0 = D_x^2 + \fa^*\fa ,\qquad \fa=\partial_y + m(y),\quad \fa^*=-\partial_y+m(y)
\end{equation}
with then $p=2$ and $q=1$. This operator is topologically trivial in the sense that $\sigma_I[H_0+Q]=0$ for $Q$ short-range \cite{bal2023topological}. Up to a constant shift, this class of operators includes $H_0=-\partial^2_x-\partial^2_y+y^2$ mentioned in the introduction.
The verification of the main hypotheses [H1]-[H2]-[H3] for that operator is carried out exactly as for the Dirac operator and will not be explored further in this paper. 

We assume that the range of $m$ is infinite, and for concreteness that:
\begin{hypothesis}\label{hyp:rangem}
We assume $m(y)-y$ is a bounded function. 
\end{hypothesis}

Under this hypothesis, the operators $H$ and $H_0$ are unbounded elliptic self-adjoint operators on the Hilbert space $L^2(\Rm^2;\Cm^2)$ with domains of definition $\mD(H)=\mD(H_0)$ the subspace of functions $(\psi_1,\psi_2)^t \in L^2(\Rm^2;\Cm^2)$  such that $\nabla \psi_j\in L^2(\Rm^2;\Cm^2) $ and $y\psi_j\in L^2(\Rm^2)$. This was denoted by the space $H^1$ with $s=0$ in \eqref{eq:Hsp}. We assume here that $Q$ decays sufficiently rapidly in $x$ as described in Hypothesis \ref{hyp:Q} for the above result to hold \cite{kato2013perturbation,RS}.  We then have the following result:
\begin{theorem} \label{thm:hypDirac}
 (i) Hypothesis [H1] holds for the unperturbed operator $H_0$; see \eqref{eq:spectralH0} and \eqref{eq:phimE}. 
 \\ \noindent (ii) Estimates (1.) and (2.) in Hypothesis \ref{hyp:H0H} hold for the Dirac operator $H_0$.
  \\ \noindent (iii) Hypothesis [H3] holds for the Dirac operator $H=H_0+Q$; see \eqref{eq:H3Dirac}. 
\end{theorem}
The proof of this theorem is presented in section \ref{sec:Dirac}.

\begin{figure}[ht!] \centering
\includegraphics[width = 0.4\linewidth]{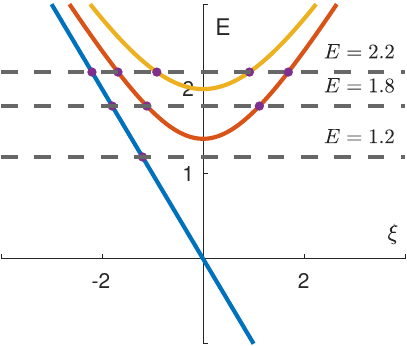}
		\caption{Branches of (absolutely continuous) spectrum of $H_0$ in \eqref{eq:DiracH0} when $m(y)=y$. When $E=2.2$, then $n_+(E)=2$, $n_-(E)=3$, while $\rME=5$. When $E=1.2$, then $n_+(E)=0$ while $\rME=n_-(E)=1$. For all choices of $[E_-,E_+]$, we have $2\pi\sigma_I[H_0]=-1$; see \cite{2, 3,quinn2021approximations}.}
    \label{fig:1}
\end{figure}

The quantization of the interface current observable $\sigma_I$ for Dirac, Klein Gordon, and more general elliptic operators with unbounded domain walls is carried out in \cite{bal2023topological}; see also \cite{3,quinn2021approximations} for the case of bounded domain walls. In this paper, we obtain the quantization of $\sigma_I$ from knowledge of the spectral decompositions of $H$ and $H_0$; see also \cite{2,3} for related spectral flow calculations. An illustration of the spectral decomposition of $H_0$ for the Dirac operator with $m(y)=y$ is given in Fig. \ref{fig:1}.

%
\section{Current conservation and interface observable}\label{sec:current}
%

This section is devoted to the proof of Theorem \ref{thm:scatmatrix}. We recall that we assume hypotheses [H1]-[H2]-[H3] to hold. Let $\psi_m$ and $\psi_n$ two generalized eigenfunctions in $H^p_{-s}$ for $s>\frac12$, solutions of
\[
  H\psi_m = E_m \psi_m,\quad H\psi_n = E_n \psi_n
\]
with $E_n$ and $E_m$ in $\Rm$. Define the current correlation
\begin{equation}\label{eq:currentcorrelation}
  J_{mn}(x_0) = (\psi_n, 2\pi i[H,P(\cdot-x_0)] \psi_m ). 
\end{equation}
Here, $(\cdot,\cdot)$ formally denotes the inner product on $\mH$. While $\psi_n\not\in \mH$, the above integral is well-defined since $[H,P(\cdot-x_0)]$ is a differential operator with coefficients that vanish for $x-x_0$ outside of a compact set and hence mapping $H^p_{-s}$ to $L^2_t$ for any $t\in\Rm$. We recall that $P$ is a switch function in $\fS[0,1]$. On that compact set in the $x-$variable,  both $H\psi_m$ and $\psi_n$ are square integrable.

While the edge transport described by the continuous spectrum of $H$ is one-dimensional, the physical setting is still genuinely two-dimensional. We thus need a notion of conserved current along the $x-$axis playing a similar role to that of, say, Wronskians in the analysis of one dimensional Sturm-Liouville operators \cite{teschl2014mathematical}. The above current correlation plays such a role.
\begin{lemma}\label{lem:currentconservation}
 When $E_m=E_n$, we have the current conservation
 \[
   J'_{mn}(x_0) = 0 \quad \mbox{ for all } \quad x_0\in\Rm.
 \]
\end{lemma}
\begin{proof}
   Using that $P'(\cdot-x_0)$ (unlike $P(\cdot-x_0)$) is compactly supported, we obtain that 
   \begin{align*}
     &-J_{mn}'(x_0) = (\psi_n, 2\pi i[H,P'(\cdot-x_0)] \psi_m ) = (\psi_n, 2\pi i (HP'(\cdot-x_0)-P'(\cdot-x_0) H)\psi_m) \\
     = \ & (H\psi_n,  2\pi i P'(\cdot-x_0)\psi_m) - (\psi_n, 2\pi i P'(\cdot-x_0) H\psi_m)  = (E_n-E_m) (\psi_n, 2\pi i P'(\cdot-x_0) \psi_m),
   \end{align*}
   which vanishes.
\end{proof}

Consider $H=H_0+Q$ with $Q$ rapidly decaying at infinity as described in the preceding section. For a fixed energy $E\in \Rm$, the unperturbed solutions of $(H_0-E)\psi=0$ in $H^p_{-s}$ are given by $\psi_m(x,y)$ for $1\leq m\leq \rME$ in \eqref{eq:unperturbedE} while the corresponding perturbed solutions are given by $\psi^Q_m(x,y)$.  The number of propagating modes $\rME$ equals $n_++n_-$ where $n_\pm$ corresponds to the number of currents $\pm J_m>0$ associated to each unperturbed plane wave and defined in \eqref{eq:currents}. From assumption \eqref{eq:approxscattering}, we deduce that
\begin{equation}\label{eq:psimV}
  (\psi_m^Q , 2\pi i[H,P(\cdot-x_0)] \psi_n^Q ) \approx \dsum_{1\leq p,q \leq \rME} \alpha^\pm_{mp}\bar\alpha^\pm_{nq} ( e^{i \xi_p x}\phi_p,  i[H,P(\cdot-x_0)] e^{i \xi_q x}\phi_q)
\end{equation}
where $\approx$ here is the same sense as above but now as $x_0\to\pm\infty$. All we need in the sequel is in fact that \eqref{eq:psimV} holds rather than the slightly more constraining \eqref{eq:approxscattering}. All terms in \eqref{eq:psimV} are again clearly defined since $[H,P(\cdot-x_0)]$ is compactly supported in the $x-$vicinity of $x_0$. We wish to estimate the above right-hand side.
\begin{lemma}\label{lem:unperturbedcurrent}
  Let $P$ be a switch function in $\fS(0,1)$. 
  Then we have 
  \begin{equation}
    ( e^{i \xi_m x}\phi_m, i[H,P] e^{i \xi_n x}\phi_n) = \delta_{mn} \partial_\xi E_n(\xi_n) = \delta_{mn} J_n.
  \end{equation}
\end{lemma}
\begin{proof}
  For an operator $A$, we denote by $A(x,x',y,y')$ its Schwartz kernel. We use the notation $A(z,y,y')$ for $z=x-x'$ when the operator invariant by translation in $x$.
 Let us first assume $\xi_m\not=\xi_n$. Then
\begin{align*}
  &(  i[H,P] e^{i \xi_m x}\phi_m , e^{i \xi_n x}\phi_n )  = \dint e^{-i\xi_n x} \phi_n^*(y) ( i[H_0,P])(x,x',y,y') e^{i\xi_m x'}\phi_m(y') dxdy dx' dy' \\
  =& \dint e^{-i\xi_n z} \phi_n^*(y)  iH_0(z,y,y') \big( (P(x')-P(x'+z)) e^{i(\xi_m-\xi_n) x'} \big) \phi_m(y') dz dx' dy dy' \\
  =&  \hat P(\xi_n-\xi_m) \dint  \phi_n^*(y)  (e^{-i\xi_n z} - e^{-i\xi_m z})  iH_0(z,y,y')\phi_m(y') dz dy dy' \\
  =& \hat P(\xi_n-\xi_m) \dint  \phi_n^*(y) (\hat H_0(\xi_n,y,y')-\hat H_0(\xi_m,y,y'))\phi_m(y') dy dy' \\
  =& \hat P(\xi_n-\xi_m) (E_n(\xi_n)-E_m(\xi_m)) (\phi_n,\phi_m) = 0
\end{align*}
since $E=E_n(\xi_n)=E_m(\xi_m)$ while $\hat P$, the Fourier transform of $P-\frac12$, is bounded for $\xi_n\not=\xi_m$ (decomposing $P$ as a Heaviside function plus an integrable function while $\hat P$ would equal $(\xi_m-\xi_n)^{-1}$ for $P$ the Heaviside function). Note that we may not have (and do not have in practice) $(\phi_n,\phi_m)=0$ for $n\not=m$ since $\xi_n\not=\xi_m$ for a fixed value of $E$ while the eigenfunctions $\phi_n$ are orthogonal for different values of $E_m$ at a fixed value of $\xi$.

When $\xi_n=\xi_m$, we find instead
\begin{align*}
  &(    i[H,P] e^{i \xi_m x}\phi_m, e^{i \xi_m x}\phi_n)  = \dint e^{-i\xi_m x}\phi_n^*(y) ( i[H_0,P])(x,x',y,y') e^{i\xi_m x'}\phi_m(y') dxdy dx' dy' \\
  =& \dint e^{-i\xi_m z} \phi_n^*(y)   iH_0(z,y,y') \big( P(x')-P(x'+z)  \big) \phi_m(y') dz dx' dy dy' \\
  =& \dint  \phi_n^*(y)  e^{-i\xi_m z} (-z)   iH_0(z,y,y')\phi_m(y') dz dy dy' \\
  =& \dint  \phi_n^*(y) \partial_{\xi}\hat H_0(\xi_m,y,y') \phi_m(y') dy dy' = (\phi_n,\partial_{\xi}\hat H_0(\xi_m)\phi_m) .
\end{align*}
The modes $\phi_m(\xi)$ satisfy
\[
  \hat H_0(\xi) \phi_m(\xi) = E_m(\xi) \phi_m(\xi).
\]
Since the spectral branches $\xi\mapsto E_m(\xi)$ are assumed sufficiently smooth, this yields
\[
  \partial_\xi \hat H_0 \phi_m + \hat H_0 \partial_\xi \phi_m = \partial_\xi E_m \phi_m + E_m \partial_\xi \phi_m
\]
from which we deduce 
\[
  (\phi_n,\partial_\xi \hat H_0 \phi_m) = \partial_\xi E_m (\phi_n,\phi_m)
\] 
for any $\phi_n$ such that $(\hat H_0-E_m)\phi_n=0$. If $n\not=m$ while $E_m(\xi_m)=E_n(\xi_m)$, then we may choose the eigenvectors $\phi_n$ and $\phi_m$ as orthogonal so that $(\phi_n,\partial_\xi \hat H_0 \phi_m)=0$ (we have that $\partial_\xi E_m\not=0$ since $E\not\in Z$ is not at a critical value of the energy branches).

As a result, when $\xi_n=\xi_m$ we have
\begin{align*}
  ( e^{i \xi_m x}\phi_n,i[H,P] e^{i \xi_m x}\phi_m)   =\delta_{mn} (\phi_m,\partial_{\xi}\hat H_0(\xi_m)\phi_m)  = \delta_{mn} \partial_\xi E_m .
\end{align*}
We used the normalization $\|\phi_m\|^2=1$.  This concludes the derivation.
\end{proof}

We thus conclude from \eqref{eq:approxscattering} and the above lemma that in the limits $x_0\to\pm\infty$,
\begin{equation}\label{eq:decpsimn}
 (\psi_m^Q,2\pi i[H,P(\cdot-x_0)]\psi_n^Q) \approx \sum_p J_p \alpha^\pm_{mp} \bar\alpha^\pm_{np}.
\end{equation}

Next we have the following result:
\begin{lemma}\label{lem:unitaryscattering}
 The scattering matrix $S(E)$ defined in \eqref{eq:scatteringmatrix} is unitary. 
\end{lemma}
\begin{proof}
  From \eqref{eq:decpsimn} evaluated at $x_0\to\pm\infty$ and the current conservation in Lemma \ref{lem:currentconservation} stating that both limits are equal, we deduce that when $J_m>0$ and $J_n>0$, then
\[
 \sum_{J_p>0} \bar T^+_{mp} T^+_{np} = \delta_{mn}  -  \sum_{J_p<0} \bar R^+_{mp} R^+_{np} .
\]
This shows the orthonormality of the first $n_+$ columns of $S$. Considering the other cases $\pm J_m>0$ and $\pm J_n>0$ provides the other orthonormality constraints and concludes the proof. 
\end{proof}
\begin{lemma} \label{lem:tracescattering}
   Let $S$ be the above scattering matrix. Then
\begin{equation}\label{eq:tracescattering}
  \trr\ T^*_+T_+ - \trr\ T^*_-T_- = n_+-n_-.
\end{equation}
\end{lemma}
\begin{proof}
 From the unitarity of the scattering matrix, we obtain that
\[
  S^*S=  \begin{pmatrix}T_+^* T_+ +R_+^*R_+ & T_+^* R_-+R^*_+ T_-  \\ R_-^*T_+ + T_-^*R_+ & R^*_- R_- + T_-^* T_-  \end{pmatrix}  = S S^* =  \begin{pmatrix}T_+ T_+^* +R_-R_-^* & T_+ R_+^*+R_- T_-^*  \\ R_+T_+^*+ T_- R_-^*  & R_+ R_+^*+ T_- T_-^*-\end{pmatrix}  = I.
\]
Looking at the diagonal terms, we obtain
\[
  \trr\  T_+^* T_+ +R_+^*R_+ = \trr\ T_+ T_+^* +R_-R_-^* = n_+,\quad
  \trr\  R^*_- R_- + T_-^* T_-  = \trr\ R_+ R_+^*+ T_- T_-^* = n_-.
\]
By cyclicity of the trace or explicit computation of the norm, we deduce that 
$\trr R_+^*R_+ = \trr R_- R_-^*= \trr R_-^* R_-$
so that $\trr T_+^* T_+ - \trr T_-^* T_- = n_+-n_-.$
This may be written using only one-sided measurements as
$n_+-n_- =  \trr\ (T_+^* T_  + +  R^*_-R_-)- n_- = n_+ - \trr\ (T_-^* T_ - +  R^*_+R_+ )$.
\end{proof}

We deduce from the spectral theorem and the decomposition \eqref{eq:spectraldecH} the following result.
\begin{proposition}\label{prop:sigmaIpsimQ}
Let $E\in \Rm\backslash Z$ and $\psi_m^Q$ the associated perturbed generalized eigenfunctions. Then:
\begin{equation}\label{eq:sigmaIpsimQ}
  \dsum_m   \Big| \frac{\partial\xi_m}{\partial E} \Big| (\psi_m^Q , 2\pi i[H,P] \psi_m^Q ) = n_+-n_-.
\end{equation}
\end{proposition}
\begin{proof}
Using the inverse function theorem, we denote $J_m^{-1}=\partial_E\xi_m$.
We have from the above calculations in \eqref{eq:decpsimn} and sending $x_0\to+\infty$
\begin{align*}
   &\dsum_m   \Big| \frac{\partial\xi_m}{\partial E} \Big| (\psi_m^Q , 2\pi i[H,P] \psi_m^Q )  =  \dsum_m   \Big| \frac{\partial\xi_m}{\partial E} \Big| \dsum_n |\alpha^+_{mn}|^2 J_n
   \\=&\dsum_{J_m>0} \frac1{|J_m|} \dsum_{J_n>0} |T^+_{mn}|^2 |J_m|+ \dsum_{J_m<0} \frac1{|J_m|} \Big(J_m + \dsum_{J_n>0} |R^+_{mn}|^2 |J_m|\Big)
   \\=& \sum_{J_m>0,J_n>0} |T^+_{mn}|^2  - n_-  +  \sum_{J_m<0,J_n>0} |R^-_{mn}|^2  = n_+-n_-.
\end{align*}
We use here that there are $n_+$ modes with $J_m>0$ and $n_-$ modes with $J_m<0$. 
\end{proof}

As a corollary of the preceding proposition, we now conclude the proof of Theorem \ref{thm:scatmatrix}:\begin{proof} [Theorem \ref{thm:scatmatrix}]
From \eqref{eq:spectraldecH}, we have by spectral calculus
\[
  \varphi'(H)= \dsum_n  \varphi'(\lambda_n) \Pi_n + \dsum_j \dint_{\Rm} \varphi'(E_j(\xi)) \Pi^Q_j(\xi) d\xi.
\]
The operators $i[H,P]\Pi_n$ are trace-class with vanishing trace since
\[
  \Tr [H,P]\Pi_n = (\phi_n,[H,P]\phi_n) = (\phi_n, HP \phi_n) - (\phi_n, PH \phi_n) = \lambda_n (\phi_n, P \phi_n) - \lambda_n (\phi_n, P \phi_n)=0.
\]
Since the sum over $n$ is at most countable (and in fact finite since $[E_-,E_+]\cap Z=\emptyset$), it does not contribute to $\sigma_I[H]$. Thus, from the definition of the rank-one projectors $\Pi^Q_j$, and identifying $\psi^Q_j(\xi)$ with $\psi^Q_m(E)$ when $E_j(\xi_m)=E$, we find
\[
\begin{array}{rcl}
  2\pi \sigma_I[H] &=& \dsum_j \dint_{\Rm} \varphi'(E_j(\xi))  (\psi^Q_j(\xi),2\pi i[H,P]\psi^Q_j(\xi))  d\xi\\
    &=& \dsum_m \dint_{\Rm} \Big|\frac{\partial\xi_m}{\partial E} \Big| \varphi'(E)  (\psi^Q_m(E),2\pi i[H,P]\psi^Q_m(E))  dE \\
    &=&  \dint_{\Rm}\varphi'(E) (n_+-n_-)dE = n_+-n_- = \trr\ T^*_+T_+ - \trr\ T^*_-T_-.
  \end{array}
\]
We used $\varphi\in \fS[0,1]$ and Lemma \ref{lem:tracescattering} to conclude. 
\end{proof}

\section{Spectral analysis and limiting absorption principle}\label{sec:lap}
This section is devoted to the proof of Theorem \ref{thm:lap}. 
We assume that $H_0$ is a self-adjoint differential operator as described in [H1] above. The resolvent operator $R_0(z)=(H_0-z)^{-1}$ then (possibly) displays different behaviors as $z$ approaches the real-axis with positive or negative imaginary parts. 

We recall that the spaces $L^2_s$ and $H^p_s$ are introduced in Definition \ref{def:LHs} and that $Z$ is the set of critical values of branches of spectrum of $H_0$ defined in \eqref{eq:Z}.  

\begin{lemma}\label{rmk:cpt}
Let $p>p'\geq0$ and $s>s'$. Then the injection  $H_s^p\xhookrightarrow{} H_{s'}^{p'}$ is compact.
\end{lemma}
\begin{proof} 
  This result follows from standard embedding results for weighted Sobolev spaces in, e.g., \cite[Chapter 6]{triebel-2006}. Defining weights $w_{p,s}(x,y)=\aver{x}^s\aver{y}^p$ satisfying the hypotheses of \cite[Definition 6.1]{triebel-2006} with $w_{p',s'}/w_{p,s} (x,y)\to0$ as $|(x,y)|\to\infty$, and with the identification $F^p_{22}(\Rm^n,w_{p,s})\cong H^p_s$, we then use \cite[Theorems 6.5 \& 6.7]{triebel-2006} to conclude.
\end{proof}
This generalizes compactness results obtained in \cite{ASNSP_1975_4_2_2_151_0}. 

We assume that hypotheses \ref{hyp:Q} and \ref{hyp:H0H} hold. The proof of Theorem \ref{thm:lap} is split into three parts.
\begin{proposition}\label{prop:eigenvalue}
The point spectrum of $H$ is discrete away from $Z$. The only possible limiting points of families of eigenvalues are in  $Z\cup \{\pm \infty\}$.
\end{proposition}
\begin{proof}
We will show that if $u\in H^p$ is an eigenfunction corresponding to $\lambda$, i.e., $Hu=\lambda u,\ a<\lambda <b$, where $a,b$ satisfies 
condition (2.) in Hypothesis \ref{hyp:H0H},
then $u\in H^p_s$ for some $s>0$ and 
\begin{gather}\label{ieq1:eigen}
    \|u\|_{H^p_s}\leq C\|u\|_{L^2},
\end{gather}
with some constant $C$ independent of $\lambda$.

The proposition is a direct corollary. Indeed, suppose $\{u_n\}$ is a set of eigenfunctions with norm 1 in $H_s^p$. By \eqref{ieq1:eigen}, $\|u_n\|_{L^2}$ is bounded below by a positive constant. Also by Lemma \ref{rmk:cpt}, the injection map from $H_s^p$ into $L^2$ is compact. As $\{u_n\}$ is orthogonal and bounded both below and above in $L^2$, it must be in a finite set. This implies that $H$ has finitely many eigenvalues in $[a,b]$ including multiplicity.

To prove \eqref{ieq1:eigen}, we first show that it holds true for $s=0$. Indeed, since $Hu=\lambda u$,
\[(H_0-i) u = (\lambda-i) u + g \in L^2,\quad g=-Qu= (H_0-\lambda)u.\]
Thus $u=R_0(i)[(\lambda-i) u + g]$ with $R_0(i)$ mapping from $\mH$ to $H^p$ by ellipticity assumption [H1(o)] and with $Q$ bounded implies that \eqref{ieq1:eigen} holds with $s=0$.  Now, the operator of multiplication by $Q(x,y)$ is a continuous operator from $L^2$ (and hence $H^p$) into $L^2_h(\mathbb{R}^2)$ by hypothesis \ref{hyp:Q}. This implies
\begin{gather*}
    \|g\|_{L_h^2}\leq C_1\|u\|_{H^p} \leq C_2\|u\|_{L^2},
\end{gather*}
with some constants $C_1, C_2$ independent of $\lambda$ on a compact interval.
For $g\in L_h^2$, we apply condition (2.) in Hypothesis \ref{hyp:H0H} to deduce that $u\in H^p_{h-1-\epsilon}$ for any $\epsilon>0$. Choosing $0<\epsilon<h-1$, we proved \eqref{ieq1:eigen}.  
\end{proof}

\begin{lemma}\label{lmm:hp}
Assume that condition (2.) in Hypothesis \ref{hyp:H0H} holds for $H_0$. Then there exists a constant $C(R)>0$ depending on $R>0$ such that the following inequality holds:
\begin{gather*}
    \|u\|_{H^p(|x|\leq R)}\leq C(R)\left( \|u\|_{H^{p-1}(|x|\leq R+1)}+ \|H_0u\|_{L^2(|x|\leq R+1)}\right). 
\end{gather*}
\end{lemma}
\begin{proof} This is a standard regularity result. We use the notation $x\in\Rm^2$ here.
We choose a $v_{R}(x)\in C_c^\infty(\mathbb{R})$ such that
\begin{gather*}
    v_{R}(x)=\begin{cases}
        1, & |x|\leq R \\
        0, & |x|\geq R+1,
    \end{cases} \qquad
    \underset{1\leq j\leq p}{\max}\|v^{(j)}_{R}(x)\|_\infty\leq C_1,
\end{gather*}
where $C_1$ is a constant independent of $R>0$.
By condition (2.) in Hypothesis \ref{hyp:H0H},
\begin{gather}\label{est:hp1}
    \|u\|_{H^p(|x|\leq R)}\leq \|v_{R} u\|_{H^p}\leq \|(H_0-\lambda)(v_{R}u)\|_{L_{1+\epsilon}^2}\leq \|H_0(v_{R}u)\|_{L_{1+\epsilon}^2}+\|\lambda u\|_{L_{1+\epsilon}^2(|x|\leq R)}
\end{gather}
for some $\epsilon>0$ and some $\lambda\in (a,b)$ where $[a,b]\cap Z=\emptyset$. Since $H_0$ is an elliptic operator of order $p$, $H_0(v_{R,r}u)$ can be represented in the form
\begin{gather}\label{eq:h0u}
H_0(v_{R}u)=v_{R}H_0u+\sum_{j=1}^pv_{R}^{(j)}(x)H'_{p-j}(u),
\end{gather}
where $H'_j$ is a differential operator of order at most $j$. Thus we have 
\begin{gather}\label{est:hp2}
\|H_0(v_Ru)\|_{L_{1+\epsilon}^2}\leq C(R)\left[\|H_0u\|_{L^2(|x|\leq R+1)}+\|u\|_{H^{p-1}(|x|\leq R+1)}\right]
\end{gather}
with some constant $C(R)>0$. We conclude the proof of the lemma using \ref{est:hp1} and \ref{est:hp2}.
\end{proof}

\begin{lemma}\label{lmm:fourier_transform}
Suppose that for some $u\in H_s^p$ with $s>\frac{1}{2}$,  we have
$$(H_0-z)u=f.$$
Then, for any $g\in L_s^2$, $(u,g)$ can be represented as follows:
\begin{gather}
    (u, g)=\sum_{j\in J}\Big(\frac{\hat{f}_{j}(\xi)}{E_j(\xi)-z}, \hat{g}_j(\xi)\Big),
\end{gather}
where we used the Fourier transform defined in \eqref{eq:FT}.
\end{lemma}
\begin{proof}
 From the spectral representation \eqref{eq:spectraldecH0} and the spectral theorem, we have
 \[
   (H_0-z)^{-1} = \dsum_{j\in J} \dint_{\Rm} \frac{\Pi_j(\xi) d\xi}{E_j(\xi)-z}
 \]
 and the formula follows from the Fourier transform \eqref{eq:FT} for $f\in L^2_s$ and the Parseval equality for the standard Fourier transform.
\end{proof}

We use the notation ${\mathcal B}(H_1,H_2)$ to denote bounded linear operators from $H_1$ to $H_2$.
We now prove the Principle of limiting absorption for $H_0$: 
\begin{theorem}[Principle of limiting absorption for $H_0$] \label{thm:pla}
Let $R_0(z)=(H_0-z)^{-1}$ be the resolvent of $H_0$. Let $a,b\in\Rm$ such that $[a,b]\cap Z=\emptyset$. For 
$\lambda\in (a,b)$, the following two limits exist in the uniform operator topology of ${\mathcal B}(L_s^2, H_{-s}^p)$:
\begin{gather*}
    \lim_{z\to\lambda\pm 0i} R_0(z)=R_0^\pm(\lambda).
\end{gather*}
Moreover, $u^\pm=R_0^\pm(\lambda)f$ are (outgoing and incoming) solutions of the equation 
\begin{gather*}
    (H_0-\lambda)u(x)=f(x).
\end{gather*}
\end{theorem}
\begin{proof}
We mainly follow the structure of \cite[Theorem 4.1]{ASNSP_1975_4_2_2_151_0}.
We prove the theorem for $R^+_0(\lambda)$ as the proof for $R^-_0(\lambda)$ is similar.
Let $f, g\in L_s^2$. We first show that the function 
$$F(z)=(R_0(z)f, g),$$
which is an analytic function of $z$ in $\rJ(a,b)$, has continuous boundary values on both edges of $(a, b)$. It suffices to prove this for $f, g$ in a dense subset $C_c^\infty(\mathbb{R}^2)$ of $L_s^2$. 

By Lemma \ref{lmm:fourier_transform}, $(R_0(z)f, g)$ can be represented as
\begin{gather}\label{eq:inner}
    (R_0(z)f, g)=\sum_{j\in J}\left(\frac{\hat{f}_{j}(\xi)}{E_j(\xi)-z}, \hat{g}_j(\xi)\right).
\end{gather}
By condition (ii.) in Hypothesis \ref{hyp:H1}, at most finitely many branches $j\mapsto E_j(\xi)$ cross $\lambda\in\Rm$. We denote the finite set consisting of these indices $j$ by $M_0$ (of cardinality ${\rm M}(\lambda)$) and $\xi_j$ to be the roots such that 
$$E_j(\xi_j)=\lambda\ \text{for}\  j\in M_0.$$
For $j\in J \setminus M_0$ and $\xi\in\mathbb{R}$, $|E_j(\xi)-\lambda|$ is uniformly bounded below by a positive constant, which can be inferred from condition (ii.) in Hypothesis \ref{hyp:H1}. Thus we take the limit $z\to\lambda+0i$ in the integral in the last term in  for $j\in J \setminus M_0$. 
\begin{gather}
    \lim_{z\to\lambda+0i}\sum_{j\in J \setminus M_0}\Big(\frac{\hat{f}_{j}(\xi)}{E_j(\xi)-z}, \hat{g}_{j}(\xi)\Big)
    =\sum_{j\in J \setminus M_0} \Big(\frac{\hat{f}_{j}(\xi)}{E_j(\xi)-\lambda}, \hat{g}_{j}(\xi)\Big)
\end{gather}
For $j\in M_0$, we have 
\begin{gather}
\lim_{z\to\lambda+0i}\left(\frac{\hat{f}_{j}(\xi)}{E_j(\xi)-z}, \hat{g}_{j}(\xi)\right)=\frac{\pi i}{a_j}\hat{f}_{j}(\xi_j)\bar{\hat{g}}_j(\xi_j)+p.v.\left(\frac{\hat{f}_{j}(\xi)}{E_j(\xi)-\lambda}, \hat{g}_{j}(\xi)\right), 
\end{gather}
as an application of the Plemelj formula, where $a_j=\lim_{\xi\to\xi_j}\frac{E_j(\xi)-\lambda}{\xi-\xi_j} = \partial_\xi E_j(\xi_j)\in\mathbb{R}\backslash0$ exists 
and the last integral is defined and being taken in the principal value sense. Combining two limits together, we obtain that \eqref{eq:inner} admits a continuous boundary value at $\lambda+0i$:
\begin{gather*}
    \lim_{z\to\lambda+0i}\aver{R_0(z)f, g}=\sum_{j\in M_0} i\pi\frac{\hat{f}_j(\xi_j)\bar{\hat{g}}_j(\xi_j)}{a_j}
    + \sum_{j\in J} p.v.\Big(\frac{\hat{f}_{j}(\xi)}{E_j(\xi)-\lambda}, \hat{g}_{j}(\xi)\Big),
\end{gather*}
where the above terms for $j\in J\backslash M_0$ are standard (Lebesgue) integrals.
Note that
\[
  \Im  \lim_{z\to\lambda\pm0i}\aver{R_0(z)f, f}= \sum_{j\in M_0} \frac{\pm\pi}{\partial_\xi E_j(\xi_j)} |\hat f_j(\xi_j)|^2.
\]

This in turn means that for any $f\in L_s^2$, there exists a weak limit of $R_0(z)f$ in $L_s^2$ as $z\to\lambda+0i$. Moreover, since the unit ball in $H_{-s}^p$ is weakly compact, we may also assume that this limit is the weak limit of $R_0(z)f$ in $H_{-s}^p$. We denote this limit by $u$ and we define the operator $R_0^+(\lambda)$ by
\begin{gather}\label{def:R_0}
    R^+_0(\lambda)f=u.
\end{gather}
By Lemma \ref{rmk:cpt}, $H_{-s}^p\xhookrightarrow{} H_{-s-1}^{p-1}$ is compact, and $\{R_0(z)f\}_{z\in \rJ_+(a,b)}$ is bounded in $H_{-s}^p$ by condition (1.) in the Hypothesis \ref{hyp:H0H}, so $u$ is also the limit in $H_{-s-1}^{p-1}$. By Lemma \ref{lmm:hp}, $R_0(z)f$ converges locally to $u$ in $H^p$, which implies that
\begin{gather}\label{eq:h0}
    (H_0-\lambda)R^+_0(\lambda)f=f.
\end{gather}
Moreover, since $\{R_0(z)\}_{z\in \rJ_+(a,b)}$ is bounded in ${\mathcal B}(L_s^2, H_{-s}^p)$, then indeed $R^+_0(\lambda)\in{\mathcal B}(L_s^2, H_{-s}^p)$.

For $f\in L_s^2$, we wish to show next that
\begin{gather}\label{claim:r0}
    R_0(z)f\to R^+_0(\lambda)f\ \text{in}\ H_{-s}^p.
\end{gather}
To see this, we first observe that for $u\in H^p_{-s}$,
\begin{gather}\label{eq:norm}
    \|u\|_{H_{-s}^p}=\|\aver{x}^{-s}u\|_{H^p}\leq \|(H_0-i)(\aver{x}^{-s}u)\|_{L^2}\leq \|H_0(\aver{x}^{-s}u)\|_{L^2}+\|u\|_{L_{-s}^2}.
\end{gather}
By \eqref{eq:h0u}, $H_0(\aver{x}^{-s}u)=\aver{x}^{-s}H_0u+\sum_{j=1}^p\frac{d^j}{dx^j}\aver{x}^{-s}\cdot H'_{p-j}(u)$,  where $H^j$ is a differential operator of order at most $j$. Thus
$
    \|H_0(\aver{x}^{-s})u\|_{L^2}\leq \|H_0u\|_{L_{-s}^2}+C_0\|u\|_{H_{-s-1}^{p-1}}
$
where $C_0>0$ is independent of $u$. Together with \eqref{eq:norm}, it follows that
\begin{gather}\label{ineq:u}
    \|u\|_{H_{-s}^p}\leq \|u\|_{L_{-s}^2}+\|H_0u\|_{L_{-s}^2}+C_0\|u\|_{H_{-s-1}^{p-1}}.
\end{gather}

Now if \eqref{claim:r0} is not true, then there is $\epsilon>0$ and a sequence $\{z_n\}$ with $z_n\to\lambda+0i$ such that
\begin{gather}\label{assm:epsilon1}
    \|R_0(z_n)f-R^+_0(\lambda)f\|_{H_{-s}^p}>\epsilon.
\end{gather}
Noting that $\{R_0(z_n)f\}$ is bounded in $H_{-s'}^p$ for $\frac{1}{2}<s' < s$ and using the compactness of the injection maps $H_{-s'}^p\xhookrightarrow{} L_{-s}^2$ and $H_{-s}^p\xhookrightarrow{} H_{-s-1}^{p-1}$ by Lemma \ref{rmk:cpt}, we can choose a subsequence of $\{z_n\}$ (still denoted by $\{z_n\}$) so that $R_0(z_n)f\to R_0^+(\lambda)f$ in $L_{-s}^2$ and in $H_{-s-1}^{p-1}$. Since $$H_0R_0(z_n)f=f+z_nR_0(z_n)f,\ H_0R_0^+(\lambda)f=f+\lambda R^+_0(\lambda)f,$$
it follows from \eqref{ineq:u} that $R_0(z_n)f\to R^+_0(\lambda)f$ in $H_{-s}^p$, which contradicts with \eqref{assm:epsilon1}. This concludes the claim \eqref{claim:r0}.

Finally we want to show that $R_0(z)\to R^+_0(\lambda)$ in the uniform topology of ${\mathcal B}(L_s^2, H_{-s}^p)$. Supposing this is not true, there exists $\epsilon>0$, a sequence $z_n\to\lambda+0i$, and a sequence $f_n$ with $\|f_n\|_{L_s^2}=1$ and $f_n\rightharpoonup f$ in $L_s^2$ such that
\begin{gather}\label{assm:epsilon}
    \|R_0(z_n)f_n-R_0^+(\lambda)f_n\|_{H_{-s}^p}> \epsilon.
\end{gather}
Indeed, for any $g\in L_s^2$, we have
$
    \lim_{n\to\infty}\aver{R_0(z_n)f_n, g}=\lim_{n\to\infty}\aver{f_n, R_0(\bar{z}_n)g}=(f,R^-_0(\lambda)g)=(R^+_0(\lambda)f, g).
$
Thus, 
\begin{gather}\label{lim:weak}
    R_0(z_n)f_n\rightharpoonup R^+_0(\lambda)f\ \text{in}\ L_{-s}^2.
\end{gather}
$\{R_0(z_n)f_n\}$ is bounded in $H_{-s'}^p$ for $\frac{1}{2}<s' <  s$, using the compactness of $H_{-s'}^p$ into $L_{-s}^2$ and $H_{-s}^p$ into $H_{-s-1}^{p-1}$, we may assume that $R_0(z_n)f_n$ converges in $L_{-s}^2$ and $H_{-s-1}^{p-1}$. It follows from \eqref{ineq:u} that $R_0(z_n)f_n$ converges in $H_{-s}^p$. Hence by \eqref{lim:weak}, $R_0(z_n)f_n$ indeed converges to $R^+_0(\lambda)f$. We can use the same argument to obtain that $R_0^+(\lambda)f_n$ also converges to $R^+_0(\lambda)f$, which contradicts \eqref{assm:epsilon}. This concludes the result that $R_0(z)\to R^+_0(\lambda)$ in the uniform topology of ${\mathcal B}(L_{-s}^2, H_{-s}^p)$.
\end{proof}

We now prove that outgoing (or incoming) solutions of $Hu=\lambda u$ for $\lambda\in\Rm$ are in fact eigenvectors of $H$.
\begin{lemma}\label{lmm:outgoing}
Assume that $Q$ satisfies \eqref{q}in Hyp. \ref{hyp:Q}.
Let $u=R^+_0(\lambda)f\in H^p_{-s}$ (resp. $u=R^-_0(\lambda)f$) for $f\in L_s^2$ satisfy
$
    Hu=\lambda u.
$
Then $u\in L^2$ and in fact $u\in L^2_t$ for any $t\in\Rm$.
\end{lemma}
\begin{proof}
We prove the lemma for $u=R^+_0(\lambda)f$; the proof for $u=R^+_0(\lambda)f$ is similar. First note that 
$
    f=(H_0-\lambda)u=(H-Q-\lambda)u=-Qu.
$
By Theorem \ref{thm:pla}, $\underset{z\to\lambda+0i}{\lim}R_0(z)f=u$ in $H_{-s}^p$, so that
$$\underset{z\to\lambda+0i}{\lim}\aver{R_0(z)f, f}=-(u, Qu).$$

By Lemma \ref{lmm:fourier_transform}, 
\begin{gather*}
    (R_0(z)f, f)=\sum_{j\in J}\left(\frac{\hat{f}_{j}(\xi)}{E_j(\xi)-z}, \hat{f}_{j}(\xi)\right).
\end{gather*}
Using the same argument as in the proof of Theorem \ref{thm:pla}, we obtain that 
\begin{gather}\label{lim:r0}
    \underset{z\to\lambda+0i}{\lim}(R_0(z)f, f)=i\pi\sum_{j\in M_0}\frac{|\hat{f}_{j}(\xi_j)|^2}{a_j}
    + \sum_{j\in J} p.v.\Big(\frac{\hat{f}_{j}(\xi)}{E_j(\xi)-\lambda}, \hat{f}_{j}(\xi)\Big)
    = -(u, -Qu).
\end{gather}
Since $Q$ is Hermitian by the Hypothesis \ref{hyp:Q}, $(u, -Qu)$ is a real number. Therefore, 
\begin{gather}\label{eq:nulhatfj}
    \hat{f}_{j}(\xi_j)=0\ \text{for}\ j\in M_0.
\end{gather}

Writing $(H_0-\lambda)u=f$ in the Fourier domain with \eqref{eq:unperturbedxi} yields
\begin{gather*}
\hat{u}_{j}(\xi) \phi_j(y,\xi) =\frac{\hat{f}_{j}(\xi) \phi_j(y,\xi)}{E_j(\xi)-
    \lambda},\ j\in M_0.
\end{gather*}
For $j\in J\setminus M_0$, $\left|\frac{1}{E_j(\xi)-\lambda}\right|$ is uniformly bounded above, so $\widehat{u}_{j}(\xi)\in L^2$ and since $f\in L_s^2$, 
\begin{gather}\label{ineq:mm}
    \sum_{j\in J \setminus M_0} \|\widehat{u}_{j}(\xi)\|^2_{L^2}<\infty.
\end{gather}
For $j\in M_0$, and since $\partial_\xi E_j(\xi)\not=0$ by assumption that $\lambda\not\in Z$ and $\hat{f}_j(\xi_j)=0$ by \eqref{eq:nulhatfj}, we apply \cite[Proposition 4.1]{yamada} (see also \cite[Theorem 3.2]{ASNSP_1975_4_2_2_151_0}) to deduce that the inverse Fourier transform of $\widehat{u}_{j}(\xi)\phi_j(y;\xi)$ is in $L_{s-1}^2$.   Combined with \eqref{ineq:mm}, we thus proved that $u\in L_{s-1}^2$. Then $f=-Qu\in L_{h+s-1}^2$. From the Hypothesis \ref{hyp:H1}, $h>1$, so we lifted the weight of the H\"older space where $f$ is in by a positive constant $h-1>0$. Repeating this procedure finitely many times implies that $u\in L^2$, and in fact in $L^2_t$ for all $t\in\Rm$ so that eigenvectors in fact decay faster than algebraically as $|x|\to\infty$. This concludes the proof of the lemma.
\end{proof}

Now we are able to prove the principle of limiting absorption for $H=H_0+Q$. We recall that $Z_H$ is the union of the (discrete) critical set $Z$ and the point spectrum $\{\lambda_n\}_n$ of $H$, which is discrete away from $Z$. 
\begin{theorem}[Principle of limiting absorption] \label{thm:pla1}
Let $R(z)=(H-z)^{-1}$ be the resolvent of $H$. Let $a,b\in\Rm$ such that $[a,b]\cap Z_H=\emptyset$. For $1<2s<h$ and $\lambda\in (a,b)$, the following limits exist in the uniform operator topology of ${\mathcal B}(L_s^2, H_{-s}^p)$:
\begin{gather*}
    \lim_{z\to\lambda\pm 0i} R(z)=R^\pm(\lambda).
\end{gather*}
Moreover, for all $f\in L_s^2$, $u^\pm=R^\pm(\lambda)f$ are solutions of the equation 
\begin{gather*}
    (H-\lambda)u(x)=f(x).
\end{gather*}
\end{theorem}
\begin{proof}
We closely follow the proof of \cite[Theorem 4.2]{ASNSP_1975_4_2_2_151_0}, the main difference being our use of Lemma \ref{lmm:outgoing}. We prove the theorem for $R^+(\lambda)$ as the proof for $R^-(\lambda)$ is similar. For $z\in \overline{\rJ_+(a,b)}$, we define $T(z)$ by
\begin{gather}
     H^p_{-s} \ni  u \mapsto T(z)u=\begin{cases}
    R_0^+(z)Qu, & \text{if $z\in (a,b)$},\\
    R_0(z)Qu, & \text{if Im z}>0.
    \end{cases}
\end{gather}
By assumption \ref{hyp:Q}, $Q$ maps $H^p_{-s}$ to $L^2_{s+\eps}$ for some $\eps>0$ when $1<2s<h$ while $R_0^+(z), R_0(z)\in \mB(L^2_{s+\eps},H^p_{-s-\eps})$ so that by Lemma \ref{rmk:cpt}, $T(z)$ is continuous on $\overline{\rJ_+(a,b)}$ in the uniform topology of ${\mathcal B}(H_{-s}^p, H_{-s}^p)$ and in fact compact from $H_{-s}^p$ to $H_{-s}^p$ for all $z\in \overline{\rJ_+(a,b)}$.

Now we claim that $(I+T(z))^{-1}$ exists for all $z\in \overline{\rJ_+(a,b)}$. For $z\in \rJ_+(a,b)$, by the Fredholm alternative for compact operators \cite[Theorem 6.6]{Brezis2010FunctionalAS}, this is equivalent to the surjectivity of $(I+T(z))$. For $f\in L^2$, $u=R(z)f$ is equivalent to
$$(I+T(z))u=R_0(z)f.$$
As an operator on $L^2$, then $\text{Ran}\ R_0(z)=H^p$ so that $H^p\subset \text{Ran}\ (I+T(z))$ and hence $\text{Ran}\ (I+T(z))=\overline{\text{Ran}\ (I+T(z))}=H_{-s}^p$. It follows that $(I+T(z))^{-1}$ exists in ${\mathcal B}(H_{-s}^p, H_{-s}^p)$.

Let $z=\lambda\in (a,b)$. By the Fredholm alternative $I+T(\lambda)$ is invertible if and only if $-1$ is not an eigenvalue of $T(\lambda)$. Suppose on the contrary that $-1$ is an eigenvalue of $T(\lambda)$ and let $u\in H_{-s}^p$ be the corresponding eigenfunction. It follows that $u=-R^+_0(\lambda)Qu$, which implies that $Hu=\lambda u$. By Lemma \ref{lmm:outgoing}, this implies that $u\in L^2$ so $\lambda$ is an eigenvalue of $H$, which is a contradiction since $[a,b]\cap Z_H=\emptyset$. Thus we conclude that $I+T(\lambda)$ is invertible.

Now since $T(z)$ is continuous on $\overline{\rJ_+(a,b)}$ in the uniform topology of ${\mathcal B}(H_{-s}^p, H_{-s}^p)$, it follows that the operator $(I+T(z))^{-1}$ is also continuous on $\overline{\rJ_+(a,b)}$ in the uniform operator topology of ${\mathcal B}(H_{-s}^p, H_{-s}^p)$. Note that
\begin{gather*}
    R(z)=(I+T(z))^{-1}R_0(z),\ z\in \rJ_+(a,b).
\end{gather*}
Since $R_0(z)$ is also continuous on $\overline{\rJ_+(a,b)}$ in the uniform operator topology of ${\mathcal B}(H_{-s}^p, H_{-s}^p)$ by Theorem \ref{thm:pla}, we have for $\lambda\in (a,b)$ that
$
    \lim_{z\to\lambda+0i} R^+(z)=(I+T(\lambda))^{-1}R^+_0(\lambda)
$
in the uniform topology of ${\mathcal B}(L_s^2, H_{-s}^p)$. Finally, letting $z\to\lambda +0i$ in 
$R(z)+R_0(z)QR(z)=R_0(z)$, we obtain that 
$$R^+(\lambda)+R_0^+(\lambda)QR^+(\lambda)=R^+_0(\lambda),$$
and it follows that $(H-z)R^+(\lambda)f=f$.
\end{proof}
\begin{remark}[Short-range perturbation] \label{rem:sr} In the above proofs, we only used the property of $Q$ that it maps $L^2_{-s}$ to $L^2_{s+\eps}$ continuously for $s\in\Rm$ (and in fact for some $\frac12<s<s_0$ to obtain that $u\in L^2$ in Lemma \ref{lmm:outgoing}). Such perturbations are called {\em short-range}.
\end{remark}

Our final result is the following:
\begin{theorem}\label{thm:absolute_cont}
$H$ does not have singular continuous spectrum.
\end{theorem}
\begin{proof}
The proof is an adaptation of, e.g., \cite[Corollary 4.2]{yamada}. Let $E(\lambda)$ be the right-continuous resolution of the identity associated with the self-adjoint operator $H$. It suffices to show that $\lambda\mapsto (E(\lambda)f,f)$ for $f\in L_s^2$ is continuous when $\lambda\in\mathbb{R}$ is not an eigenvalue.  
First we assume that $$k_n<\alpha<\beta<k_{n+1}$$ for some $n\in\mathbb{Z}$, and $[\alpha,\beta]$ does not contain any eigenvalue, where $\{k_n\}$ label elements in $Z_H$, the union of $Z$ defined in \eqref{eq:Z} and the discrete spectrum of $H$. 
For arbitrary $a,b$ with $\alpha<a<b<\beta$, we have the following relation using, e.g., \cite[(5.32)]{kato2013perturbation}, 
\begin{gather*}
\frac{1}{2}[(E(a)+E(a-0))f,f]-\frac{1}{2}[(E(b)+E(b-0))f,f]\\
=\frac{1}{2\pi i}\ \underset{\eta\to 0}{\lim}\int_a^b((H-\mu-i\eta)^{-1}f-(H-\mu+i\eta)^{-1}f,f)d\mu,
\end{gather*}
so that
\begin{gather}\label{eaeb}
|(E(a)f,f)-(E(b)f,f)|=\frac{1}{2\pi i}\int_a^b(u^+(\mu,f)-u^-(\mu,f),f)d\mu.
\end{gather}
Then we deduce from Theorem \ref{thm:pla} that
$
    |(E(a)f,f)-(E(b)f,f)|\leq C(b-a)\|f\|_{L_s^2},
$
where $C$ is a constant only depending on $\alpha,\beta$ and $s$, which shows the desired absolute continuity on $(\alpha,\beta)$. Finally, since $Z_H$  is a discrete set and the singular continuous spectrum cannot be supported on a discrete set, we conclude that $H$ does not have any singular continuous spectrum.
\end{proof}

\section{Generalized eigenfunction expansion} \label{sec:eigexp}

This section presents a proof of Theorem \ref{thm:EET}, based on the results obtained in Theorem \ref{thm:lap}, and in particular the construction of generalized eigenfunctions $\psi_j^Q(\cdot;\xi)$ for $\fco=(j,\xi)\in J\times \Rm$. From Theorem \ref{thm:lap},  $R(z)=(H-z)^{-1}$ is well defined on $\mH$ and bounded uniformly for $z\in \rJ(a,b)$ away from $Z_H$ so that we have the bounded operators:
\[R^\pm(\lambda)=(H-(\lambda\pm i0))^{-1}.\]  

For $\fco=(j,\xi)\in J\times\Rm$, we defined $\psi_j(x,y;\xi)$ in \eqref{eq:unperturbedxi}.  For $z\in \rJ(a,b)$, we recall that $A_\fco(z)=(I-R(z)Q)\psi_j(\xi)$  is defined \eqref{eq:Az}  and that for 
$f\in L^2_s$ with $s>\frac12$, $A^*_\fco f(z) =(f,A_\fco(z))$ in \eqref{eq:Azstar}. This allows us to define $\psi_j^\pm(\xi) = A_\fco(E_j(\xi)\pm i0) = (I-R(E_j(\xi)\pm i0)Q)\psi_j(\xi)$ in \eqref{eq:perturbedxi} and for concreteness, $\psi_j^Q(x,y;\xi)=\psi_j^+(x,y;\xi)$, the {\em outgoing} generalized eigenfunctions, while $\psi_j^-(x,y;\xi)$ corresponds to {\em incoming} generalized eigenfunctions.

Thanks to Theorem \ref{thm:pla}, we observe that $\psi_j^\pm(\xi)\in H^p_{-s}(\Rm^2)$. For each $E\in(a,b)$ with $[a,b]\cap Z_H=\emptyset$, there is a finite number of wavenumbers $\xi_m$ such that $E=E_m(\xi_m)$. We denote by $\psi_m^\pm(E)$ the corresponding generalized eigenfunctions parametrized by $(m,E)$ rather than $(j,\xi)$.

 Let $f\in L_s^2$ and $j\in J$ with $\xi\in \Xi_j$ so that $E_j(\xi)\in (E_-,E_+)\backslash Z_H$. We define as in \eqref{def:fourier}-\eqref{def:fourierdisc} the generalized Fourier transform(s)
\[
    \tilde{f}^\pm(\fco)=A_\fco^*(E_j(\xi)\pm 0i)f,\qquad \tilde{f}_n = (f,\phi_n).
\]
We first verify the following estimate on the generalized Fourier transform:
\begin{lemma}\label{lmm:a_star}
For $\frac{h}{2}>s>\frac{1}{2}$, there exists a constant $C=C(s,a,b)$ such that
\begin{gather}
    |A^*_\fco(z)f|\leq C\|f\|_{L_s^2}
\end{gather}
for all $\fco\in J\times\mathbb{R}$, $z\in \rJ(a,b)$ and $f\in L_s^2$.
\end{lemma}
\begin{proof}
By construction, we have
\begin{gather}
    A^*_\fco(z)f= (f, A_\fco(z))_{L^2}=(f, (I-R(z)Q)\psi_{j}(\xi))_{L^2}.
\end{gather}
Multiplication by $Q$ is a continuous map from $L_{-s}^2$ to 
$L_{h-s}^2$
and from Theorem \ref{thm:pla1}, we get
\begin{gather*}
   |(f, (I-R(z)Q)\psi_{j}(\xi))_{L^2}|\leq\|f\|_{L_s^2}\cdot\|(I-R(z)Q)\psi_{j}(\xi))\|_{L_{-s}^2}\\
   \leq C\|f\|_{L_s^2}\cdot\|\psi_{j}(\xi)\|_{L_{-s}^2}\leq C\Big(\int_{\Rm} \aver{x}^{-2s}dx\Big)^{\frac{1}{2}}\|f\|_{L_s^2}
\end{gather*}
and hence the result.
\end{proof}
We next state the following properties of the resolvent operator:
\begin{proposition}
\label{prop:(h-z)a}
For $s>\frac{1}{2}$, $\operatorname{Im}z\neq 0$ and $f\in L_s^2$, we have
\begin{gather}
 \label{eq:(h-z)a}
    (H-z)A_\Xi(z)=(E_j(\xi)-z)\psi_j(\xi),\\
\label{hat_r(z)f}
    \widehat{R(z)f}(\Xi)=\frac{A^*_{\Xi}(\bar{z})f}{E_j(\xi)-\bar{z}},
\end{gather}
where $\hat{f}(\Xi)=(f,\psi_j(\xi))$ is the unperturbed Fourier transform defined in \eqref{eq:FT}. 
\end{proposition} 
\begin{proof}
We have by construction
$A_\fco(z) = (I-R(z)Q)\psi_j(\xi)$
so that 
\[ (H-z) A_\fco(z) = (H_0-z) \psi_{\fco} = (E_j(\xi)-z) \psi_j(\xi).\]
This is \eqref{eq:(h-z)a}.
Thus, from the definition of the (unperturbed) Fourier transform \eqref{eq:FT},
\[ \Big(f,\frac{(H-z) A_\fco(z)}{E_j(\xi)-z}\Big) = \hat f(\fco).\]
From this, we deduce
\[  \widehat{R(\overline{z})f}(\fco) = \Big( R(\bar z)f,\frac{(H-z) A_\fco(z)}{E_j(\xi)-z}\Big) =  \Big(f,\frac{A_\fco(z)}{E_j(\xi)-z}\Big) = \frac{A^*_\fco(z) f}{E_j(\xi)-\bar z}.\]
This proves the result.
\end{proof}

By proposition \ref{prop:eigenvalue}, we know that $H$ has discrete point spectrum away from the discrete set $Z_H$. Let $\{\varphi_n\}$ be the countable set of orthonormalized eigenfunctions of $H$. We now prove Theorem \ref{thm:EET}, which for convenience, we recall here:
\begin{theorem}[Eigenfunction Expansion Theorem]
For $f\in L_s^2\subset\mH$,  we define $\tilde f$ as either one of the generalized Fourier transforms $\tilde f^\pm$ in \eqref{def:fourier}. Then we have the following Parseval relation:
\[
    \|f\|_{L^2}=\sum_{n=0}^\infty |(f, \varphi_n)|^2+\int_\mathbb{R}\sum_{j\in J}|\tilde{f}(\Xi)|^2d\xi.
\]
\end{theorem}
\begin{proof}
For fixed $\xi$, let 
$$J_{(\alpha,\beta)}(\xi)=\{j\in J |   \ E_j(\xi)\in (\alpha,\beta)\}.$$
Let $[\alpha,\beta]$ contain no eigenvalue of $H$. We wish to show that
\begin{gather}\label{first_assertion}
    ((E(\beta)-E(\alpha))f,f)=\int_{\Rm}\sum_{j\in J_{(\alpha, \beta)}(\xi)}|\tilde{f}(\Xi)|^2 d\xi,
\end{gather}
and for an eigenvalue $\lambda$, that
\begin{gather}\label{second_assertion}
     ((E(\lambda)-E(\lambda-))f,f)=\sum_{i=1}^k|(f, \varphi_{\lambda,i})|^2,
\end{gather}
where $\{\varphi_{\lambda,i}\}$ are the orthonormalized eigenfunctions of $H$ associated with the eigenvalue $\lambda$.

The second assertion \eqref{second_assertion} reduces to the well-known Parseval equality of Fourier series. To prove the first assertion, we first assume that 
$k_n<\alpha<\beta<k_{n+1}$
for some $n\in\mathbb{Z}$ where $\{k_n\}$ label elements in $Z_H$. We make use of
\begin{gather*}
    ((E(\beta)-E(\alpha))f,f)=\frac{1}{2\pi i}\underset{\eta\downarrow 0}{\mathlarger{\lim}}\int_\alpha^\beta (R(\mu+i\eta)f-R(\mu-i\eta)f,f)d\mu.
\end{gather*}
Since $(E(\lambda)f, f)$ is absolutely continuous with respect to $\lambda$, we have using the resolvent equation $R(z_1)-R(z_2)=(z_1-z_2)R(z_1)R(z_2)$ and \eqref{hat_r(z)f}, that
\begin{gather}\label{eq:e}
\begin{aligned}
   & ((E(\beta)-E(\alpha))f,f)=\frac{1}{2\pi i}\underset{\eta\downarrow 0}{\lim}\int_\alpha^\beta 2i\eta (R(\mu-i\eta)R(\mu+i\eta)f,f)d\mu\\
    =&\frac{1}{\pi }\underset{\eta\downarrow 0}{\lim}\int_\alpha^\beta \eta \|R(\mu+i\eta)f\|^2d\mu=\frac{1}{\pi }\underset{\eta\downarrow 0}{\lim}\int_\alpha^\beta \eta \|\widehat{R(\mu+i\eta)f}\|^2d\mu\\
   =&\frac{1}{\pi }\underset{\eta\downarrow 0}{\mathlarger{\lim}}\int_\alpha^\beta\int_{\mathbb{R}}\sum_{j}\eta\left|\frac{A^*_\Xi(\mu-i\eta)}{E_j(\xi)-\mu +i\eta}f\right|^2d\xi d\mu,
\end{aligned}
\end{gather}
using the Parseval identity for $\mF$ in \eqref{eq:FT} (i.e., $\mF$ is an isometry) and \eqref{hat_r(z)f} for the last line. To analyze the above integral, we recall  (see \cite[Proposition 4.3]{yamada1975eigenfunction}) that for $f(\mu, \eta)$  a continuous function on $[\alpha, \beta] \times [0, \eta_0]\ (\alpha<\beta, \eta_0>0)$, we have
\begin{gather}\label{eq:limits}
    \frac{1}{\pi}\underset{\eta\downarrow 0}{\lim}\int_\alpha^\beta\frac{\eta}{(\lambda-\mu)^2+\eta^2}f(\mu,\eta)d\mu=\begin{cases}
0, &  \lambda>\beta\ or\  \lambda<\alpha  \\
\frac12 f(\lambda,0), & \lambda=\alpha\ or\ \beta\\
f(\lambda,0), & \alpha<\lambda<\beta.
\end{cases}
\end{gather}
Moreover, the above integral is uniformly bounded on $[\alpha, \beta] \times [0, \eta_0]$.

Lemma \ref{lmm:a_star} implies that $|A^*_\Xi(\mu-i\eta)f|^2=|\tilde f(\Xi)|^2$ is continuous in $\Xi$ and uniformly bounded  on $[\alpha, \beta]\times [0,\eta_0]$. Then \eqref{eq:limits} yields
\begin{gather}\label{lim:eta}
\begin{aligned}
    &\underset{\eta\to 0}{\lim} \frac1\pi\int_\alpha^\beta\eta\left|\frac{A^*_\Xi(\mu-i\eta)}{E_j(\xi)-\mu-i\eta}f\right|^2d\mu
    \\& =\underset{\eta\to 0}{\lim}\frac1\pi\int_\alpha^\beta\frac{\eta}{|E_j(\xi)-\mu|^2+\eta^2} |A^*_\Xi(\mu-i\eta)f|^2 d\mu
    \ =\ \begin{cases}
    |\tilde{f}(\Xi)|^2, &  j\in J_{(\alpha,\beta)}(\xi) \\
    \frac12 |\tilde{f}(\Xi)|^2 &  E_j(\xi)=\alpha\  \text{or}\ \beta \\
    0, &  \text{otherwise}.
    \end{cases}
\end{aligned}
\end{gather}
Moreover, this integral is uniformly bounded for $\Xi\in\mathbb{R}\times J$ and $(\mu,\eta)\in (\alpha,\beta)\times(0,1)$. 

We split $J$ into two parts $J_1=\big\{j\in J|\  d(E_j(\xi),[\alpha,\beta]) \geq 1, \ \forall \xi\in \Rm\big\}$ and $J_0:=J\setminus J_1$.
Here, $d(x,X)$ is Euclidean distance between a point $x\in\Rm$ and an interval $X\subset \Rm$.
By hypothesis \ref{hyp:H1}, $J_0$ is finite. Let $j\in J_0$. Using \eqref{lim:eta} for $\xi$ such that $E_j(\xi)\in[\alpha-1,\beta+1]$, and $|E_j(\xi)-\mu+i\eta|^2\geq C(1+|E_j(\xi)|^2)$ otherwise, we deduce that there is a constant $C=C(\alpha,\beta)$ such that:
\begin{gather}\label{est:uniform}
\sum_{j\in J_0} \int_\alpha^\beta\eta\left|\frac{A^*_\Xi(\mu-i\eta)}{E_j(\xi)-\mu+i\eta}f\right|^2d\mu \ \leq\    \frac{C \|f\|_{L^2_s}}{1+|E_j(\xi)|^2} .
\end{gather}

For $j\in J_1$, we have by construction that $|E_j(\xi)-\mu-i\eta|\geq 1$. 
Notice that
\begin{gather*}
    \sum_{j\in J}|A^*_\Xi(z)f|^2=\sum_{j\in J}|(f, A_\Xi(z))_{L^2}|^2=\sum_{j\in J}|(f, (I-R(z)Q)\psi_j(\xi))_{L^2}|^2\\
    =\sum_{m\in M}|((I-R(\bar{z})Q)f, \psi_j(\xi))_{L^2}|^2=\|\widehat{(I-R(z)Q)f\ }(\xi,y)\|_{L^2(y)}^2.
\end{gather*}
Since the operator $f\mapsto R(z)Qf$ from $L^2_s$ to $L^2_{-(s+h)}$ is uniformly bounded for $z\in J(\alpha,\beta)$, we deduce that
$$\sum_{j\in J_1}\int_\alpha^\beta\eta\left|A^*_\Xi(\mu-i\eta)f\right|^2d\mu$$ 
is also uniformly bounded for $\xi\in\mathbb{R}$ and $(\mu,\eta)\in (\alpha,\beta)\times(0,1)$. Thus we conclude that there exists a constant $C=C(\alpha,\beta)$ independent of $\xi,\mu,\eta$ such that
\begin{gather}\label{est:uniform2}
 \sum_{j\in J_1}\int_\alpha^\beta\eta\left|\frac{A^*_\Xi(\mu-i\eta)}{E_j(\xi)-\mu+i\eta}f\right|^2d\mu\leq\frac{C\|f\|_{L^2_s}}{1+|E_j(\xi)|^2}. 
\end{gather}
The estimates in \eqref{est:uniform} and \eqref{est:uniform2} are integrable in $\xi$ by Hypothesis \ref{hyp:H1}.
Thus, by the dominated convergence theorem, we deduce from \eqref{eq:e} and \eqref{lim:eta} that
\begin{gather}\label{end}
\begin{aligned}
    &((E(\beta)-E(\alpha))f,f)
   =\frac{1}{\pi }\int_{\mathbb{R}}\underset{\eta\downarrow 0}{\mathlarger{\lim}}\int_\alpha^\beta\sum_{j\in J}\eta\left|\frac{A^*_\Xi(\mu-i\eta)}{E_j(\xi)-\mu-i\eta}f\right|^2d\mu d\xi
   \\ =&  \ \int_{\Rm}\sum_{j\in J{(\alpha, \beta)}(\xi)}|\tilde{f}(\Xi)|^2d\xi.
\end{aligned}
\end{gather}
This proves the first assertion \eqref{first_assertion} when $k_n<\alpha<\beta<k_{n+1}$
for some $n\in\mathbb{Z}$. By monotone convergence, we deduce that 
\[
  ((E(k_{n+1}-)-E(k_{n}+))f,f) = 
   \int_{\Rm}\sum_{j\in J_{(k_n, k_{n+1})}(\xi)}|\tilde{f}(\Xi)|^2d\xi.
\]
This handles the absolutely continuous part of the spectrum of $H$. It remains to address the discrete set of points $k_n$, which as in Theorem \ref{thm:absolute_cont}, carries only point spectrum. This is taken care of by the first term on the right-hand side in \eqref{eigenfunction_expansion} as in \eqref{second_assertion}.
This concludes the proof of the theorem.
\end{proof} 

\section{Application to Dirac operators with domain walls}\label{sec:Dirac}

We now prove Theorem \ref{thm:hypDirac}. For the Dirac operator $H=H_0 + Q$ with $Q$ the operator of multiplication by $Q(x,y)$, which takes values in $2\times 2$ Hermitian matrices, it is convenient to recast $H$ as $UHU^*$ with $U$ a unitary matrix so that $UHU^*$, still called $H$, takes the form $H=H_0 + Q$, where
\begin{gather}\label{eq:H0}
    H_0=D_x\sigma_3-D_y\sigma_2+m(y)\sigma_1=\left(\begin{array}{cc}
-i\partial_x & \fa \\
\fa^*& i\partial_x\\
\end{array}\right),\qquad \fa := \partial_y+m(y).
\end{gather}
We recognize in $\fa$ the annihilation operator of the quantum harmonic oscillator when $m(y)=y$. 

The theory developed in section \ref{sec:current} requires that we prove the hypotheses [H1](o-iv), and in particular the spectral decompositions in \eqref{eq:spectraldecH0} and \eqref{eq:spectraldecH}. 
\subsection{Spectral decomposition of the unperturbed operator} \label{sec:decH0}
Since $H_0$ is invariant with respect to translations in the $x-$variable, we introduce the partial Fourier transform $H_0=\mF^{-1}_{\xi\to x} \hat H_0(\xi) \mF_{x\to\xi}$ with 
\begin{gather}\label{fourier:h0}
\hat H_0(\xi)=\xi\sigma_3+i\partial_y\sigma_2+m(y)\sigma_1
 =\left(\begin{array}{cc}
             \xi & \fa \\
             \fa^* & -\xi\\
\end{array}\right).
\end{gather}
We then compute
\begin{gather}\label{delta}
H_0^2-\lambda^2 = (H_0+\lambda)(H_0-\lambda)=\left(\begin{array}{cc}
             -\partial_{xx}+\fa\fa^* - \lambda^2 & 0 \\
             0 &  -\partial_{xx}+\fa^*\fa - \lambda^2\\
             \end{array}\right).
\end{gather}
We now present a spectral decomposition of $H_0$ as well as an explicit expression for the resolvent operator $R_0(z)=(H_0-z)^{-1}$ for $z\in \Cm$ with $\Im z\not=0$.

We first observe that since the range of $m(y)$ is unbounded, standard results on Sturm-Liouville operators \cite{teschl2014mathematical} show that $\fa^*\fa$ and $\fa\fa^*$ have a compact resolvent and hence discrete spectrum with simple eigenvalues. It is also straightforward to observe that $\fa$ admits a kernel in $L^2(\Rm)$ of dimension one. The positive eigenvalues of $\fa^*\fa$ and $\fa\fa^*$ are the same and we thus get the existence of simple eigenvalues $\rho_0=0<\rho_n<\rho_{n+1}$ for $n\geq1$ with $\rho_n/n$ converging to $2$ as $n\to\infty$. Moreover, we have the existence of two $L^2(\Rm;dy)-$orthonormal bases $(\nu_n)_{n\geq0}$ and $(\mu_n)_{n\geq1}$ such that 
\begin{equation}\label{eq:varphipsin}
    \fa^*\fa \nu_n = \rho_n \nu_n,\quad n\geq0;\qquad \fa\fa^* \mu_n=\rho_n \mu_n,\quad n\geq1.
\end{equation}
When $m(y)=y$, then $\nu_n$ and $\mu_n$ are both Hermite functions associated to the quantum harmonic oscillator. We also verify that (for $\|\cdot\|$ the standard $L^2(\Rm;dy)$ norm)
\begin{equation}\label{eq:controlfa}
\|f\| + \|yf\| + \|D_y f\| \leq C \|\fa^* f \|,\qquad \|yf\| + \|D_y f\| \leq C (\|\fa f\| + \|f\|).
\end{equation}
The asymmetry between the above two results stems from the fact that $\fa^*$ has trivial $L^2-$kernel while $\fa\nu_0=0$.

Define $\Pi^\nu_n=\nu_n\otimes \nu_n$ for $n\geq0$ and $\Pi^\mu_n=\mu_n\otimes \mu_n$ for $n\geq1$ (with $\Pi^\mu_0=0$ to simplify notation). Then, we observe that we have the spectral decomposition:
\begin{equation}
    H_0^2 -\lambda^2 = \mF_{\xi\to x}^{-1}  \dsum_{n\geq0} \dint_{\Rm} d\xi \ 
    (\xi^2+\rho_n-\lambda^2) \begin{pmatrix} \Pi^\nu_n & 0 \\ 0 & \Pi^\mu_n \end{pmatrix} \, \mF_{x\to\xi}.
\end{equation}
This provides the following explicit expression for the resolvent
\begin{eqnarray}
    &R_0(\lambda) &= (H_0-\lambda)^{-1} = (H_0+\lambda) (H_0^2 -\lambda^2)^{-1} 
     \\ && = (H_0+\lambda) \mF_{\xi\to x}^{-1}  \dsum_{n\geq0} \dint_{\Rm} d\xi \ 
    (\xi^2+\rho_n-\lambda^2)^{-1} \begin{pmatrix} \Pi^\nu_n & 0 \\ 0 & \Pi^\mu_n \end{pmatrix} \, \mF_{x\to\xi}.
\end{eqnarray}
Estimates on $R_0(\lambda)$ may thus be obtained by applying $H_0+\lambda$ to the resolvent of the operator $H_0^2$. The above construction also shows that the eigenvalues of $\hat H_0^2(\xi)$ are given explicitly by $E_n^2(\xi)= \xi^2+\rho_n$ for $n\geq0$. We now show that $\pm E_n(\xi)$ are indeed eigenvalues of $\hat H_0(\xi)$. To simplify notation, we introduce the set $M$ of indices $m=(\pm1,n)$ for $n\geq1$ and $0\equiv (-1,0)$ for $n=0$. We then define the eigenvectors $\phi_m(\xi)$ for $m=(\pm,n) \in M$ as 
\[  \phi_m = \begin{pmatrix} \varphi_m \\ \psi_m \end{pmatrix}, \qquad E_m = \pm (\xi^2+\rho_n)^{\frac12}\]
with $\phi_0=(\nu_0,0)^t$ independent of $\xi$ for $m=0$ and for $n\geq1$,
\begin{equation}\label{eq:phimxi}
    \phi_m(\xi) =  c_n \begin{pmatrix} (E_m(\xi)+\xi) \nu_n \\ \rho_n\mu_n \end{pmatrix}, \qquad c_n^{-2} = 2E_m(\xi)(E_m(\xi)+\xi) >0.
\end{equation}
We verify that $c_m$ is indeed defined and independent of $\pm$ and hence labeled $c_n$. Since the functions $\nu_n$ and $\mu_n$ form an orthonormal basis, we deduce that the functions $\phi_m$ also form an orthonormal basis of $L^2(\Rm;\Cm^2)$.
From this completeness result, we deduce the spectral decomposition
\begin{equation}\label{eq:spectralH0}
    H_0 = \mF_{\xi\to x}^{-1} \dsum_m \dint_{\Rm} d\xi \  E_m(\xi)\, \Pi^\phi_m(\xi)  \ \mF_{x\to\xi} ,\qquad \Pi^\phi_m(\xi) = \phi_m(\xi) \otimes \phi_m(\xi).
\end{equation}
This provides an explicit expression for \eqref{eq:spectraldecH0} in hypothesis [H1] with $\psi_m$ defined by \eqref{eq:unperturbedxi}.

\medskip

We finally turn to the construction of the generalized eigenfunctions \eqref{eq:unperturbedE} at a fixed energy level $E\in\Rm\backslash Z_D$, where for the Dirac operator, we deduce from the explicit expression in \eqref{eq:phimxi} that the set $Z$ of critical values defined in \eqref{eq:Z} is given explicitly by
\begin{equation}\label{eq:ZDirac}
  Z_D = \Big\{ \, \pm \sqrt{\rho_n}; \ n\in\Nm \ \Big\}.
\end{equation}
When $m(y)=y$ is a linear domain wall, then we verify that $\rho_n=2n$. By hypothesis \ref{hyp:rangem}, we deduce from \cite[Theorem 4.10]{kato2013perturbation} that $|\rho_n-2n|$ is bounded independently of $n$.

We construct a basis of $L^2(\Rm;\Cm^2)$ called $\phi_m(E)$ with a slight abuse of notation, solutions of $(\hat H_0(\xi_m)-E)\phi_m(E)=0$. The values $\xi_m$ are defined explicitly by
\begin{equation}\label{eq:xim}
    \xi_m = \epsilon_m (E^2-\rho_n)^{\frac12}
\end{equation}
where $m=(\epsilon_m,n)$ and $(-1)^{\frac12}=i$. Thus $\xi_m$ is real-valued for $n$ sufficiently small and purely imaginary when $E^2-\rho_n<0$. Since $E\not\in Z_D$, $E^2-\rho_n\not=0$. We then define
\begin{equation}\label{eq:phimE}
    \phi_m(E) =  c_n \begin{pmatrix} \sqrt{\rho_n} \mu_n \\ 
    (E-\xi_m(E))\nu_n \end{pmatrix}, \qquad c_n^{-2} = \rho_n+|E-\xi_m|^2>0.
\end{equation}
The functions $m\mapsto \phi_m(y;E)$ form a basis of $L^2(\Rm;\Cm^2)$ but no longer an orthonormal one. However, the orthonormalization of this basis  is a bounded operator with bounded inverse as we now show. 

When $m=(n,\epsilon_m)$ while $q=(p,\epsilon_q)$, we verify that $(\phi_m,\phi_q)=0$ when $n\not=p$. This is a direct consequence of the orthogonality of the families $\nu_n$ and $\mu_n$. However, for $q=m':=(m,-\epsilon_m)$, then we have
\[
(\phi_m,\phi_{m'})=\frac{\rho_n+(E-\epsilon_m)(\overline{E+\epsilon_m)}}{\sqrt{\rho_n+|E-\xi_m|^2}\cdot\sqrt{\rho_n+|E+\xi_m|^2}}=\frac{1+\frac{\overline{E+\xi_m}}{E+\xi_m}}{2+2\frac{|E|^2+|\xi_m|^2}{E^2-\xi^2_m}}.
\]
As a consequence, $|(\phi_m,\phi_{m'})|<\frac12$. The functions $(\phi_m)$ are thus linearly independent and form a basis of $L^2(\Rm;\Cm^2)$ by completeness of the families $\nu_n$ and $\mu_n$. The procedure of orthonormalization of the (normalized) basis elements $\phi_m$ is therefore an operator of norm bounded by $2$.

\subsection{Estimates for unperturbed operator}
We now verify that Hypothesis \ref{hyp:H0H} holds for the unperturbed Dirac operator.

\begin{proposition}\label{prop:dirac}
Estimate (1.) in Hypothesis \ref{hyp:H0H} holds for the Dirac operator.  
\end{proposition}
\begin{proposition}\label{prop:dirac_real}
Estimate (2.) in Hypothesis \ref{hyp:H0H} holds for the Dirac operator. 
\end{proposition}
The first proposition is useful for $s>\frac12$ close to $\frac12$ while the second estimate is useful for $s-1-\eps>0$.  

\begin{proof}[Proposition \ref{prop:dirac}]
We introduce
$u=(H_0+\lambda) v$ so that $v=(H_0^2-\lambda^2)^{-1}(H_0-\lambda) u$.  We wish to show that
\begin{equation}\label{eq:vdouble}
  \|v\|_{H^2_{-s}} \leq C  \|(H_0-\lambda) u \|_{L^2_s} = \|(H_0^2-\lambda^2) v \|_{L^2_s}.
\end{equation}

Recall from \eqref{delta} that
$H_0^2-\lambda^2 = {\rm Diag}( D_x^2+\fa^*\fa-\lambda^2 , D_x^2+\fa\fa^*-\lambda^2 )$, and that for orthonormal bases $\{\nu_n\},\ \{\mu_n\}$, we have $\fa^*\fa \nu_n=\rho_n \nu_n$ with $\fa\fa^* \mu_n = \rho_n \mu_n$. We also have $\fa\nu_0=0$.
Consider
\[ (D_x^2+\fa^*\fa -\lambda^2) v_1=g_1\]
with the expansion
\[ v_1=\sum_{n\geq0} v_{1n}(x) \nu_n(y) \qquad \mbox{ so that } \qquad  (D_x^2+\rho_n-\lambda^2) v_{1n} = g_{1n}\]
with obvious notation, where $\rho_n\geq0$. Define $S=S(s)=\sup_{x\in\Rm} \aver{x}^{-2s} |D_x (\aver{x}^{2s})| <\infty$.

We use \cite[Theorem A.1]{ASNSP_1975_4_2_2_151_0} for the operator $P(D)=D^2$ and for $z=\lambda^2-\rho_n$ to obtain that
\[  \|v_{1n}\|_{H^2_{-s}} \leq C \|(D_x^2+\rho_n-\lambda^2) v_{1n}\|_{L^2_s}, \]
with a constant $C=C(s,a,b)$ independent of $v_{1n}$, $\lambda\in \rJ(a,b)$ and $\rho_n\leq 2\Re\lambda^2+S$ since $z$ belongs to a compact set then. We now verify that for $\rho_n\geq 2\Re\lambda^2+S$, then
\[\|(D_x^2+\rho_n-\lambda^2) v_{1n}\|^2_{L^2_s} =\|(D_x^2-\lambda^2) v_{1n}\|^2_{L^2_s}  + \rho_n\Big( (\rho_n-2\Re\lambda^2) \|v_{1n}\|^2_{L^2_s} +(D^2_xv_{1n},v_{1n})_{L^2_s} + (v_{1n},D^2_x v_{1n})_{L^2_s}\Big).\]
We observe for $w\in C^\infty_c(\Rm)$ that 
$(D^2_xw,w)_{L^2_s} =  \|D_x w\|^2_{L^2_s} + \int_{\Rm} D_x \bar w (\aver{x}^{-2s} D_x (\aver{x}^{2s}))\aver{x}^{2s}dx. 
$
Since
\[ \Big|D_x \bar w (\aver{x}^{-2s} D_x (\aver{x}^{2s}))\aver{x}^{2s}\Big| \leq \frac12 |D_x w|^2 \aver{x}^{2s} + |w|^2 \frac S2  \aver{x}^{2s}\]
we deduce by a density argument that 
\[\|(D_x^2+\rho_n-\lambda^2) v_{1n}\|_{L^2_s} \geq \|(D_x^2-\lambda^2) v_{1n}\|_{L^2_s}\]
when $\rho_n\geq 2\Re\lambda^2+S$ so that the above bound on $\|v_{1n}\|_{H^2_{-s}}$ in fact holds with $C$ independent of $\rho_n\geq0$.
We then sum over $n$ using the orthonormality of the families $\nu_n$ and $\mu_n$ to get
\[  \|D_x^2 v_1\|_{L^2_{-s}} + \|v_1\|_{L^2_{-s}} \leq C \|(H_0^2-\lambda^2) v_1\|_{L^2_s}. \]
Now we use the relation
\[  \fa^*\fa v_1 = (H_0^2-\lambda^2)v_1 + \lambda^2v_1 - D_x^2 v_1 \]
and the above inequality to deduce from \eqref{eq:controlfa} bounds for $D_y^2 v_1$ as well $y^2v_1$ in weighted $L^2$ space. This implies that  \eqref{eq:vdouble} holds for $v_1$. 
We perform the same calculation for 
$v_2 = \sum_{n\geq1} v_{2n}(x) \mu_n(y)$.
Thus, \eqref{eq:vdouble} holds. It remains to apply $(H_0+\lambda)$ to $v$, count derivatives and powers of $y$, and obtain
\begin{equation}\label{eq:Hms}
  \|u\|_{H^1_{-s}} \leq C \|(H_0-\lambda) u \|_{L^2_s}.
\end{equation}
This concludes the derivation.
\end{proof}

While the proof of Proposition \ref{prop:dirac} was based on the spectral decomposition of $\fa^*\fa$ leading to that of $H_0^2-\lambda^2$, we now base the proof of Proposition \ref{prop:dirac_real} when $\lambda\equiv E$ is real-valued on the direct plane wave expansion of $H_0-\lambda$ using the eigen-elements $\phi_m(E)$ in \eqref{eq:phimE}.

We first need the following result on the operator $D=-i\partial_x$:
\begin{lemma}\label{lmm:1dim2}
Let $u\in H^1(\mathbb{R})$, and $s\geq0$. 

When $\lambda\in i\Rm$ and $\epsilon>0$, there is $C=C(s,\epsilon)$ such that
\begin{gather}\label{eq:lambda}
    \big\|u \big\|_{H_{s-1-\epsilon}^1}\leq C(1+|\lambda|)\Big\|\Big(\frac{d}{dx}-\lambda\Big)u\Big\|_{L^2_s}.
\end{gather}

When $\lambda\in \Rm$, there is $C=C(s)$ such that
\begin{gather}\label{eq:lambda2}
\big\|u\big\|_{L^2_s} \leq \frac{C}{|\lambda|(|\lambda| \wedge 1)^s} \Big\|\Big(\frac{d}{dx}-\lambda\Big)u\Big\|_{L^2_s},\quad 
\Big\|\frac{du}{dx}\Big\|_{L^2_s} \leq \frac{C(1+|\lambda|)}{|\lambda|(|\lambda| \wedge 1)^s} \Big\|\Big(\frac{d}{dx}-\lambda\Big)u\Big\|_{L^2_s}.
\end{gather}
\end{lemma}
\begin{proof}
    We start with $\lambda\in \Rm$. Since $u\in H^1$, we have $\underset{|x|\to\infty}{\lim}u(x)=0$ and
    \begin{gather*}
    u(x)=\int_{-\infty}^xf(t)e^{\lambda(x-t)}dt=\int_{\infty}^xf(t)e^{\lambda(x-t)}dt,
    \qquad f(x) = \Big(\frac{d}{dx}-\lambda\Big)u(x).
\end{gather*}
When $x\leq 0$,
\begin{gather*}
    |u(x)|^2\leq \int_{-\infty}^x|f(t)|^2(1+t^2)^sdt  \int_{\infty}^x(1+t^2)^{-s}dt\leq C\|f\|^2_{L_s^2} (1+x^2)^{1-2s}.
\end{gather*}
A similar estimate holds for $x\geq 0$. Multiplying with $(1+x^2)^{s-1-\epsilon}$ with $\epsilon>0$, we have $u\in L_{t}^2$ and $$\|u\|_{L_t^2}\leq C\|f\|_{L_s^2},\quad t=s-1-\epsilon.$$
Together with $u'=f+\lambda u$, we obtained the inequality \eqref{eq:lambda}.

\medskip 

Consider now the case $\lambda\in\Rm$. Denote $\partial_x\equiv\frac{d}{dx}$. The second inequality in \eqref{eq:lambda2} is a consequence of the first one and the fact that $\partial_x u =(\partial_x-\lambda)u + \lambda u$.
 We may assume $\lambda\geq0$ without loss of generality as the case $\lambda<0$ then holds after $x\mapsto -x$. Let us define $f=(\partial_x-\lambda)u$. Let $\eps>0$ and define $w_\eps(x)=\aver{\eps x}^s$. We find for $s\geq0$ that 
 \[ \Big\|\frac{w_\eps'}{w_\eps}\Big\|_\infty \leq C\eps .\]
 The result in $L^2$ (with norm $\|\cdot\|$) for $\beta=0$ holds. Indeed in the Fourier domain, 
 \[
    \hat u(\xi) = \frac1{i\xi-\lambda} \hat f,\qquad \|u\|\leq \frac1\lambda \|f\|, 
 \]
 by the Parseval equality. 
 For $s>0$ we have
 \[
    (\partial_x-\lambda) (w_\eps u) = w_\eps f - \frac{w_\eps'}{w_\eps} w_\eps u.
 \]
 Thus
 \[
    \|w_\eps u\| \leq  \frac{1}{\lambda} ( \|w_\eps f\| + C\eps \|w_\eps u\|).
 \]
 Choosing $\eps$ so that $C\eps=\frac12(\lambda\wedge 1)$, we deduce that 
 \[
    \|w_\eps u\|  \leq \frac{C}{\lambda}  \|w_\eps f\| .
 \]
 When $\lambda\gtrsim1$, we choose $\eps\sim1$ so that $w_\eps \sim \aver{x}^s$ and the result is clear. When $\lambda\lesssim1$, we use that
 \[  \lambda^s \aver{x}^s \sim \eps^{s} \aver{x}^s  \leq w_\eps(x) \leq \aver{x}^s.
 \]
 This shows that for $\lambda\lesssim1$, $\|\aver{x}^s u\| \leq C \lambda^{-1-s} \|\aver{x^s}f\|$.
\end{proof}

\begin{proof}[Proposition \ref{prop:dirac_real}.] 
    We use the basis $\phi_m(y;E)$  at a fixed $a<\lambda = E<b$ to decompose any smooth function $(x,y)$ as 
    \[ u(x,y) = \dsum_{m\in M} u_m(x) \phi_m(y;E).\] 
  We showed that $\phi_m$ formed a basis equivalent to an orthonormal one in the sense that at fixed $x$,
 \begin{equation}\label{eq:normbound}
    \|u(x,y) \|^2_{L^2_y} \approx \dsum_m  |u_m(x)|^2.
 \end{equation}
 Here $a\approx b$ when for some constant $C>0$ we have $C^{-1}a\leq b\leq Ca$.
 We also know that
 $
   (\hat H_0(\xi_m) - \lambda)\phi_m=0.
 $
 Thus,
$(H_0-\lambda) u_m\phi_m = (D_x-\xi_m) u_m (\varphi_m,0)^t - (D_x+\xi_m) (0,\psi_m)^t$
so that 
 \[
   \|(H_0-\lambda)u \|^2_{L^2_s} \approx \dsum_m \|(D_x-\xi_m) u_m\|^2_{L^2_s} + \|(D_x+\xi_m) u_m\|^2_{L^2_s}.
 \]
 Recall that $\xi_m=\eps_m(\lambda^2-\lambda_n)^{\frac12}$. 
 When $\xi_m\in\Rm$, we use \eqref{eq:lambda} to deduce that for $s-1-\epsilon\geq0$,
 \[ \|u_m\|_{H^1_{s-1-\epsilon}}  \leq C\|(D_x-\xi_m) u_m\|_{L^2_s}. \]

 When $\xi_m\in i\Rm$, we use \eqref{eq:lambda} to deduce that
 \[ 
    \|u_m\|_{H^1_s} \leq \frac{C(1+|\xi_m|)}{|\xi_m| (|\xi_m|\wedge 1)^s}  \|(D_x-\xi_m) \|_{L^2_s}.
 \]
 Since $a<\lambda<b$ and $(a,b)\cap Z_D=\emptyset$, we deduce from the definition of $\xi_m$ that $|\xi_m(\lambda)|$ is bounded below uniformly in that interval. Note, however, that $|\xi_m(\lambda)|$ tends to $0$ as $\lambda$ approaches $Z_D$.

 Summing these inequalities over $m$ and using \eqref{eq:normbound}, we deduce that
 \[
    \|u\|^2_{L^2_{s-1-\epsilon}} + \|D_x u\|^2_{L^2_{s-1-\epsilon}} \leq C \|(H_0-\lambda)u \|^2_{L^2_s}
 \]
 for $s\geq0$. The system $(H_0-\lambda)u=f$ may be recast as
 \[
   \fa^* u_2 = f_1 - (D_x-\lambda) u_1,\quad \fa u_1 = f_2 + (D_x+\lambda) u_2.
 \]
 From the properties of $\fa$ and $\fa^*$, this implies that 
 \[
    \| y u_j\|_{L^2_{s-1-\epsilon}} + \|D_y u_j \|_{L^2_{s-1-\epsilon}} \leq C \|f\|_{L^2_{s}} + \|u\|_{L^2_{s-1-\epsilon}} \leq C \|f\|_{L^2_{s}}.
 \]
 This concludes the derivation of \eqref{eq:H0real}.
\end{proof}


%
\subsection{Scattering matrix and interface current observable}\label{sec:scattering}
The objective of this section is to prove [H3] for the Dirac operator. 
Fix an energy $E\in\mathbb{R}\setminus Z_H$. 
We decompose the generalized eigenfunction $\psi^Q_m(x,y;E)$ in the basis of $\phi_q(y;E)$ in \eqref{eq:phimE} as
\begin{gather}\label{eigen}
    \psi_m^Q(x,y;E)=\sum_{q\in M}A_q(x)\phi_q(y)=\sum_{q\in M}B_q(x)e^{i\xi_qx}\phi_q(y).
\end{gather}
The decomposition consists of finitely many propagating modes $q\in M(E)$ and countably many evanescent modes in $M\backslash M(E)$. The cardinality of $M(E)$ was denoted by $\rME$ in Hypothesis [H1](iv). Note that in our setting, $\rME\geq1$ for all $E\in\Rm \backslash Z_H$.
We wish to show that in the limit $x\to\pm\infty$, $\psi^Q_m$ is well approximated as a linear combination of propagating modes. 
\begin{proposition}
The generalized eigenfunctions satisfy the following approximate decomposition:
\begin{equation}\label{eq:H3Dirac}
\psi^Q_m(x,y)\approx \sum_{q\in M(E)}\alpha_{mq}^\pm e^{i\xi_qx}\phi_q(y),
\end{equation}
with respect to norm $|||u|||=\underset{x\in\mathbb{R}}{\max}\ \|u(x,\cdot)\|_{L^2(\cdot)}$; see \eqref{eq:convergenceatinfinity} below for a more precise statement.
\end{proposition}
\begin{proof}
$\psi^Q_m(x,y;E)$ satisfies $\left(H_0-E\right)\psi^Q_m=-Q\psi^Q_m$.
We decompose $-Q\psi^Q_m$ in the basis $\phi_q(y;E)$:
$$-Q\psi^Q_m=\sum_{q\in M}a_q(x)\phi_q(y),\qquad a_q(x)=(-Q\psi^Q_m,\ \phi_q(y))_y.$$
Since $\psi^Q_m\in H_{-s}^1$ for $s>\frac{1}{2}$ and $|Q(x,y)|=O(|x|^{-h})$ for some $h>1$, then $Q\psi^Q_m\in L_t^2$ for some $t>\frac{1}{2}$ and hence $a_q(x)\in L_t^2$.

Let $p=(-\epsilon_q,n)$ be the conjugate of $q=(\epsilon_q, n)$. Then it holds that
\begin{gather}\label{claim}
    (H_0-E)\left[A_q(x)\phi_q(y)+A_p(x)\phi_p(y)\right]=a_q(x)\phi_q(y)+a_p(x)\phi_p(y).
\end{gather}
The proof is now based on a direct computation. 
Since $(\hat H_0(\xi_m)-E)\phi_m(E)=0$ and $\xi_p=-\xi_q$ we verify that 
\[
  \phi_q = \frac{E}{\xi_q}\sigma_3\phi_q+(1-\frac{E}{\xi_q})\sigma_3\phi_p\quad\mbox{and}\quad 
  \phi_p=(1+\frac{E}{\xi_q})\sigma_3\phi_q-\frac{E}{\xi_q}\sigma_3\phi_p.
\]
From the above and \eqref{claim}, we deduce that 
\begin{gather*}
     \left[-iA'_q(x)-\xi_qA_q(x)\right]\sigma_3\phi_q+\left[-iA'_p(x)-\xi_pA_p(x)\right]\sigma_3\phi_p=a_q(x)\phi_q(y)+a_p(x)\phi_p(y)\\
=a_q(x)\left[\frac{E}{\xi_q}\sigma_3\phi_q+(1-\frac{E}{\xi_q})\sigma_3\phi_p\right]+a_p(x)\left[(1+\frac{E}{\xi_q})\sigma_3\phi_q-\frac{E}{\xi_q}\sigma_3\phi_p\right].
\end{gather*}
Thus
\begin{gather*}
   B_q'(x)=ie^{-i\xi_qx} \big( -iA'_q(x)-\xi_qA_q(x)\big) = ie^{-i\xi_qx} \Big[\frac{E}{\xi_q}a_q(x)+\left(1+\frac{E}{\xi_q}\right)a_p(x)\Big].
\end{gather*}

When $\xi_q\in \mathbb{R}$, since $a_q, a_p\in L_{s}^2$ for some $s>\frac{1}{2}$, then  $a_q, a_p\in L^1$, and hence $B_q(x)$ converges to two constants as $x\to\pm\infty$, i.e., $\underset{x\to\pm\infty}{\lim}B_q(x)=\alpha_{mq}^{\pm}.$
Moreover, 
\begin{gather}\label{e1}
    |B_q(x)-\alpha_{mq}^\pm|=O(|x|^{\frac{1}{2}-t})\ \text{as}\ x\to\pm\infty.
\end{gather}

When $\xi_q\in i\mathbb{R}$, with $\xi_p=-\xi_q\in i\Rm$, then by applying Lemma \ref{lmm:1dim2}, it follows that there exists a constant independent of $q$ such that
    $
        \|A_q(x)\|_{H_t^1}\leq C\|a_q\|_{L_t^2}+C\|a_p\|_{L_t^2}
    $
    so $\lim_{|x|\to\infty}A_q(x)=0$. Moreover, $(1+|x|)^tA_q(x)\in H^1$ and by Sobolev's inequality, we have \begin{gather*}
    (1+|x|)^t|A_q(x)|\leq C_1 ( \|a_q\|_{L_t^2} + \|a_p\|_{L_t^2}).
\end{gather*}
This shows that
\begin{gather}\label{e2}
     \Big\|\sum_{q\in M \backslash M(E)}A_q(x)\phi_q(y)\Big\|_{L_y^2}\leq \frac{C_2}{(1+|x|)^t}\sum_{\xi_q\in i\mathbb{R}}\|a_q\|_{L_s^2}.
\end{gather}

Therefore, by \eqref{e1} and \eqref{e2}, we have that as $x\to\pm\infty$, 
\begin{gather}\label{eq:convergenceatinfinity}
\Big\|\psi^Q_m(x,y)-\sum_{m\in M\backslash M(E)} \alpha_{mq}^\pm\phi_q(y)\Big\|_{L_y^2}=O(|x|^{\frac{1}{2}-t})\ \text{as}\ x\to\pm\infty.
\end{gather}
This concludes the proof of the proposition.
\end{proof}
This concludes the derivation of [H3] for the Dirac operator. 

\section*{Acknowledgment} This work was supported in part by NSF Grants DMS-2306411 and DMS-1908736.


\end{document}